\documentclass[11pt]{amsart}

\usepackage[margin=1in]{geometry}  % set the margins to 1in on all sides
\usepackage{graphicx}              % to include figures
\usepackage{amsmath}               % great math stuff
\usepackage{amsfonts}              % for blackboard bold, etc
\usepackage{amsthm}      % better theorem environments
\usepackage{mathtools}          
\usepackage{color}
\usepackage{tabulary}
\usepackage{caption}
\usepackage{subcaption}
\usepackage{amssymb}
\usepackage{todonotes}
\theoremstyle{plain}
\usepackage{mathptmx}
\usepackage{tikz-cd} 
\usepackage{mathrsfs}
\usepackage{verbatim}
\usepackage{tikz}
\usepackage{enumitem}
\usepackage{hyperref}
\usepackage[hyphenbreaks]{breakurl}

\newtheorem{thm}{Theorem}[section]
\newtheorem{lem}[thm]{Lemma}
\newtheorem{prop}[thm]{Proposition}
\newtheorem{cor}[thm]{Corollary}

\newtheorem{defn}[thm]{Definition}
\newtheorem{rmk}[thm]{Remark}

\newtheorem{ques}[thm]{Question} 
\newtheorem{fact}[thm]{Fact}

\newtheorem{conv}[thm]{{\bf Convention}}

\DeclareMathOperator{\Aut}{Aut}
\DeclareMathOperator{\Out}{Out}
\DeclareMathOperator{\Inn}{Inn}

\DeclareMathOperator{\Stab}{Stab}
\DeclareMathOperator{\Fix}{Fix}

\newcommand{\NN}{\mathbb{N}}      %
\newcommand{\ZZ}{\mathbb{Z}}      % for Integers
\newcommand{\QQ}{\mathbb{Q}} 
\newcommand{\RR}{\mathbb{R}}      % for Real numbers
\newcommand{\HH}{\mathbb{H}}      % for Integers

\DeclareMathOperator{\PSL}{PSL}
\DeclareMathOperator{\GL}{GL}

\DeclareMathOperator{\Cone}{Cone}
\DeclareMathOperator{\CAT}{CAT}
\DeclareMathOperator{\Isom}{Isom}

\DeclareMathOperator{\FC}{FC}
\DeclareMathOperator{\AC}{AC}

\DeclareMathOperator{\Trip}{Trip}
\DeclareMathOperator{\Mod}{Mod}
\DeclareMathOperator{\diam}{diam}
\DeclareMathOperator{\BS}{BS}

\begin{document}

\title{On the kernel of actions on asymptotic cones}

\author{Hyungryul Baik}
\author{Wonyong Jang}
\address{Department of Mathematical Sciences, KAIST,
291 Daehak-ro, Yseong-gu, 
34141 Daejeon, South Korea}
\email{hrbaik@kaist.ac.kr \& jangwy@kaist.ac.kr}

\maketitle

%%%%%%%%%%%%%%%%%%%%%%%%%%%%%%%%%%%%%%%%%%%%%%%%%%%%%%%%%%%%%%%%%%%%%%

\begin{abstract}
 Any finitely generated group $G$ acts on its asymptotic cones in natural ways. The purpose of this paper is to calculate the kernel of such actions. First, we show that when $G$ is acylindrically hyperbolic, the kernel of the natural action on every asymptotic cone coincides with the unique maximal finite normal subgroup $K(G)$ of $G$. 
 Secondly, we use this equivalence to interpret the kernel of the actions on asymptotic cones as the kernel of the actions on many spaces at ``infinity". 
 For instance, if $G \curvearrowright M$ is a non-elementary convergence group, then we show that the kernel of actions on the limit set $L(G)$ coincides with the kernel of the action on asymptotic cones. 
 Similar results can also be established for the non-trivial Floyd boundary and the $\CAT(0)$ groups with the visual boundary, contracting boundary, and sublinearly Morse boundary. 
 Additionally, the results are extended to another action on asymptotic cones, called Paulin’s construction.
 In the last section, we calculate the kernel on asymptotic cones for various groups, and as an application, we show that the cardinality of the kernel can determine whether the group admits non-elementary action under some mild assumptions. 
%Lastly, we investigate the kernel on asymptotic cones for various groups and propose a conjecture for being acylindrically hyperbolic.
\end{abstract}
\vspace{0.6cm} 

%%%%%%%%%%%%%%%%%%%%%%%%%%%%%%%%%%%%%%%%%%%%%%%%%%%%%%%%%%%%%%%%%%%%%%

\textbf{AMS Classification numbers (2020)} Primary: 20F65, 20F67, Secondary: 20F69 , 51F30

\vspace{0.6cm} 
%%%%%%%%%%%%%%%%%%%%%%%%%%%%%%%%%%%%%%%%%%%%%%%%%%%%%%%%%%%%%%%%%%%%%%

\section{Introduction} \label{sec:intro}
 In geometric group theory, group action has proven to be a very powerful tool to study various properties of groups. Abundant examples of ``nice" actions of groups on ``good" spaces have been constructed and studied by many authors.
 Some of the classical examples include hyperbolic group actions on their Gromov boundary, Kleinian group actions on the limit sets, $\Out(F_n)$ actions on the outer spaces or free-factor complexes, mapping class group actions on the Teichm\"{u}ller spaces, and the list goes on and on.

There are rather fairly general constructions of spaces on which groups act. One of the most famous and well-studied objects is the Cayley graph which can be defined for any group. While this graph depends on the choice of a generating set, it is well-defined up to quasi-isometry whenever its generating set is finite. In this case, the Cayley graph is connected and locally finite, hence also proper. Moreover, by the construction, a group $G$ acts canonically on the Cayley graph of $G$, and this action is known to be isometric, cocompact, and properly discontinuous.
 Consequently, the Cayley graph plays a central role in the study of finitely generated groups.

 Another such example is an asymptotic cone. Roughly speaking, an asymptotic cone is a view of the metric space from infinity through our ultrafilter-glasses. 
 For any finitely generated group $G$, we can construct an asymptotic cone of $G$ together with a non-principal ultrafilter $\omega$ and a suitable real sequence $d_n$, by considering $G$ as a metric space. 
 We refer readers who are interested in the precise definitions and further details of asymptotic cones to Section \ref{Pre:Asymptotic}.
 We also refer the readers to M. Sapir's note \cite{SapirNote}, R. Young's note \cite{YoungNote}, \cite{MR3753580}, \cite[Chapter 7]{MR2007488}, etc.

 Asymptotic cones of a group reflect many properties of the group. For example, a group is of polynomial growth if and only if all of its asymptotic cones are locally compact (for a brief explanation, see \cite[Chapter 4]{MR2281936} and references therein). 
 It is also known that a finitely generated group is hyperbolic if and only if all of its asymptotic cones are real trees \cite{MR919829}. 
 Furthermore, if we restrict it to a finitely presented group, then $G$ is hyperbolic if and only if one of its asymptotic cones is a real tree \cite[Appendix]{MR2507115}.
Nevertheless, because of the unexpected and pathological phenomena of asymptotic cones, e.g., the existence of a group with 
uncountably many non-homeomorphic asymptotic cones, actions on asymptotic cones are less known than the action on the Cayley graph.

In this paper, we study the kernel of the actions of $G$ on its asymptotic cone.
To the best of our knowledge, this is the first attempt to describe the kernel.
Our first main theorem is that if $G$ is finitely generated acylindrically hyperbolic, then the kernel of the natural action on every asymptotic cone is the same as the unique maximal finite normal subgroup of $G$, denoted by $K(G)$. 

\begin{thm}[Acylindrically hyperbolic group, Theorem \ref{AHG:4-equality-CP}]
Let $G$ be a finitely generated acylindrically hyperbolic group. Then 
$$ \ker \left ( G \curvearrowright \Cone_{\omega}(G,d_n) \right ) = K(G) $$
 for all ultrafilter $\omega$ and sequence $d_n$.
\end{thm}

Then, combining with prior results obtained by Osin \cite{MR3430352}, Minasyan-Osin \cite{MR3368093}, and Dahmani-Guirardel-Osin \cite{MR3589159} (for details, see Lemma \ref{Lem_KnownResults}), it turns out that these two subgroups also coincide with subgroups $\ker \left ( G \curvearrowright \partial X \right ), \ \FC(G)$, and $\mathcal{A}(G)$.

\begin{cor} [Corollary \ref{AHG:4-equality}]
If $G$ is a finitely generated acylindrically hyperbolic group, then 
\begin{align*}
 \begin{aligned}
\ker \left ( G \curvearrowright \Cone_{\omega}(G,d_n) \right ) = \ker \left( G \curvearrowright \partial X \right ) = K(G) = \FC(G) = \mathcal{A}(G)
\end{aligned}
\end{align*}
for all ultrafilter $\omega$, sequence $d_n$, and $\delta$-hyperbolic space $X$ on which $G$ acts non-elementarily and acylindrically.
\end{cor}
For precise definitions and more details of these subgroups, we refer to Section \ref{sec:4subgroups}.
 We want to point out that a $\delta$-hyperbolic space $X$ on which $G$ acts is not uniquely determined. Indeed, we can always choose $X$ to be a quasi-tree. See \cite{MR3685605} and \cite{MR4057354}.
 Furthermore, it is known that some acylindrically hyperbolic groups have various asymptotic cones. Our proof does not depend on the choice of the ultrafilter $\omega$ and the sequence $d_n$, so it directly implies the following.

\begin{cor}[Corollary \ref{AHG:Kernel_Invariant}]
 Let $G$ be a finitely generated acylindrically hyperbolic group. Then the kernel
 $$ \ker \left ( G \curvearrowright \Cone_{\omega}(G,d_n) \right ) $$ is invariant under ultrafilter $\omega$ and sequence $d_n$.
\end{cor}

Next, we establish a connection between asymptotic cones and other spaces at ``infinity" through the kernel. More specifically, we show that the kernel of the natural actions on asymptotic cones is equivalent to the kernel of the canonical actions on various spaces at `infinity".
First, we obtain the convergence group version corollary.

\begin{cor}[Convergence group action, Corollary \ref{CG:4-equality}]
 Let $G$ be a finitely generated group. Suppose that $G$ admits a convergence group action $ G \curvearrowright M$ with $|L(G)| > 2$. Then
$$ \ker \left ( G \curvearrowright \Cone_{\omega}(G,d_n) \right ) = \ker \left( G \curvearrowright L(G) \right ) = K(G) = \FC(G) = \mathcal{A}(G) $$
 for all ultrafilter $\omega$ and sequence $d_n$.
\end{cor}
 Recall that a non-elementary convergence group is an acylindrically hyperbolic group \cite{MR3995017} and the limit set of a convergence group can be considered as the space at ``infinity". We point out that if $G$ is a non-elementary convergence group, then in general, $\partial X \neq L(G)$.
 Corollary \ref{CG:4-equality} can be obtained from the fact that the kernel of the action on its limit set is the same as $K(G)$.

 According to \cite{MR2005231}, if $G$ has a non-trivial Floyd boundary $\partial_F G$, then the action of $G$ on $\partial_F G$ is a convergence group action. All the terminology and notions about the Floyd boundary can be found in Section \ref{sec:GFB}. So, Corollary \ref{CG:4-equality} directly implies the following.

\begin{cor}[Floyd boundary, Corollary \ref{FB:4-equality}]
 Let $G$ be a finitely generated group. Suppose that $G$ has non-trivial Floyd boundary $\partial_F G$. Then 
 $$ \ker \left ( G \curvearrowright \Cone_{\omega}(G,d_n) \right ) = \ker \left( G \curvearrowright \partial_F G \right ) = K(G) = \FC(G) = \mathcal{A}(G)$$
 for all ultrafilter $\omega$ and sequence $d_n$.
\end{cor}
 
 It is a well-known fact that a $\CAT(0)$ group $G$ with a rank-one isometry is acylindrically hyperbolic, provided that $G$ is not virtually cyclic (\cite{MR4423140} and \cite{MR3849623}, see Proposition \ref{Sublinearly-Rank-one-Acylindrical} for the details. Motivated by this result, we obtain the following result.

\begin{cor}[CAT(0) group, Corollary \ref{CAT:4-equality} and Corollary \ref{CAT:boundaries}] \label{CAT:result}
Suppose that a finitely generated group $G$ acts geometrically on proper $\CAT(0)$ space $X$. Also, assume that $G$ contains a rank-one isometry and $G$ is not virtually cyclic.
 All subgroups in the following formula are the same for any ultrafilter $\omega$, sequence $d_n$, and sublinear function $\kappa$.
\begin{align*}
 \ker \left ( G \curvearrowright \Cone_{\omega}(G,d_n) \right ) &= \ker \left( G \curvearrowright \partial_{v} X \right ) = K(G) = \FC(G) = \mathcal{A} (G) \\
 &= \ker \left( G \curvearrowright \partial_{\kappa} X \right )  \\ 
 &= \ker \left( G \curvearrowright \partial_{*} X \right )
\end{align*}
Here, $\partial_v X$, $\partial_{*} X$ and $\partial_{\kappa} X$ are the visual, contracting and sublinearly Morse boundary of $X$, respectively.
\end{cor} 

 So far, we have concentrated only on the canonical action on asymptotic cones. Now we will introduce another action on asymptotic cones known as Paulin's construction.
Briefly speaking, Paulin's construction is a ``twisted" action. This action is constructed by infinitely many (outer) automorphisms so we need the condition that its outer automorphism group $\Out(G)$ should be infinite.
 Presumably, the main difference between the natural action and this new action is the existence of a global fixed point.
 It immediately follows that the natural action must have a global fixed point, but we eliminate global fixed points via the ``twisted" method. 
 Our second main result says that two kernels of these two actions on asymptotic cones coincide so 
the kernel of Paulin’s construction can also be characterized as in Corollary \ref{AHG:4-equality}.

 \begin{thm} [Theorem \ref{Kernel_Paulin}]
 Let $G$ be a finitely generated acylindrically hyperbolic group satisfying that its outer automorphism group $ \Out(G)$ is infinite.
 For any sequence $ [\phi_1] , \ [\phi_2] , \cdots $ in $\Out(G)$ with $[\phi_i]\neq[\phi_j]$, 
 there exist automorphism representatives $\phi_1 , \ \phi_2 , \cdots \in \Aut(G)$ such that the kernels of two actions on an asymptotic cone are the same, that is, 
 $$ \ker \left ( G \curvearrowright \Cone_{\omega}(G,d_n) \right ) = \ker \left ( G \overset{p}{\curvearrowright} \Cone_{\omega}(G) \right )$$ 
 for any ultrafilter $\omega$ and sequence $d_n$.
\end{thm}

 In the other direction, we concentrate on the kernel of various (non-acylindrically hyperbolic) groups and prove that if $G$ is a finitely generated infinite virtually nilpotent group, then the kernel of $G$ on its asymptotic cone is infinite.
 Using this fact, we can show that the kernel $ \ker \left ( G \curvearrowright \Cone_{\omega}(G,d_n) \right ) $ determines whether the actions are non-elementary or not, as follows.

\begin{cor} [Corollay \ref{elementary-and-infinity}]
Let $G$ be an infinite hyperbolic group. Then $G$ is non-elementary if and only if the kernel $ \ker \left ( G \curvearrowright \Cone_{\omega}(G,d_n) \right )$ is finite.
  
 Furthermore, suppose that finitely generated group $G$ acts acylindrically on some $\delta$-hyperbolic space $X$ and contains a loxodromic element. Then the action of $G$ on $X$ is non-elementary (so $G$ is acylindrically hyperbolic) if and only if $ \ker \left ( G \curvearrowright \Cone_{\omega}(G,d_n) \right )$ is finite.
\end{cor}

 Then we also study the equalities in \eqref{Main_Equalities} for various groups. For this purpose, we design the following definition.
\begin{defn} [Definition \ref{Condition ast}]
If a finitely generated group $G$ has $K(G)$ and satisfies
\begin{align*}
 \begin{aligned}
\ker \left ( G \curvearrowright \Cone_{\omega}(G,d_n) \right ) = \ker \left( G \curvearrowright \partial X \right ) = K(G) = \FC(G) = \mathcal{A}(G)
\end{aligned} \tag{$\ast$}
\end{align*}
for some $\delta$-hyperbolic space $X$, 
then we say that $G$ satisfies conidition \eqref{Main_Equalities}.
\end{defn}
 We explore condition \eqref{Main_Equalities} and in particular, we show that the braided Thompson group $BV$ satisfies condition \eqref{Main_Equalities} but it is not acylindrically hyperbolic.
 
 The paper is organized as follows. In Section \ref{sec:prelim}, we briefly recall the definition and properties of asymptotic cones, acylindrically hyperbolic groups, and related topics. 
 In Section \ref{sec:4subgroups}, we review some remarks about subgroups in Corollary \ref{AHG:4-equality}, mainly $K(G), \ \mathcal{A}(G)$, and $\FC(G)$. 
 We also include some short proofs of easily obtained inclusions and equalities in this section.
 We prove the main theorem, Theorem \ref{AHG:4-equality-CP} in Section \ref{sec:AHG}.
 As an application, we prove Corollary \ref{CG:4-equality} (a non-elementary convergence group), Corollary \ref{FB:4-equality} (a group with a non-trivial Floyd boundary), and Corollary \ref{CAT:result} (a $\CAT$(0) group with a rank-one isometry including a Coxeter group) in Section \ref{sec:App1}. 
 In Section \ref{sec:Paulin}, we characterize the kernel of Paulin's construction, and in Section \ref{sec:action}, we consider general groups and find some algebraic conditions to disturb the equivalence. In this section, we record the kernels of some interesting groups acting on their asymptotic cones.
 \vspace{0.6cm} 

\noindent \textbf{Acknowledgment} We would like to thank Inhyeok Choi for the helpful comments.
In particular, he suggested Lemma \ref{QIembeddingL} which shortened the proof of our first main theorem.
The authors would like to express deep appreciation to Anthony Genevois for a lot of invaluable comments and discussions. 
We are grateful to Lvzhou Chen for the comments on the big mapping class group and the ray graph.
The authors are grateful to Sang-hyun Kim and Mladen Bestvina for the useful discussions and suggestions.
We express our gratitude to Gye-Seon Lee for suggesting the direction for the Coxeter group, Corollary \ref{Application:Coxeter}.
We would also like to thank KyeongRo Kim, Donggyun Seo, Hongtaek Jung, and Juhun Baik for fruitful conversations.
We are also grateful to Dongha Lee for carefully reading the initial draft and for suggestions to improve the preprint.
The authors would like to thank anonymous referees for pointing out the missing references and suggestions for qualitative improvement.
This work was partially supported by the National Research
Foundation of Korea (NRF) grant funded by the Korea government (MSIT)
(No. 2020R1C1C1A01006912).

%%%%%%%%%%%%%%%%%%%%%%%%%%%%%%%%%%%%%%%%%%%%%%%%%%%%%%%%%%%%%%%%%%%%%%

\section{Preliminaries} \label{sec:prelim}

\subsection{Cayley graphs and $\delta$-hyperbolic spaces}
 In this subsection, we briefly recall the definition of Cayley graphs and $\delta$-hyperbolic spaces. We also introduce our notations.

\begin{defn}
 Let $G$ be a group with a fixed generating set $S$. The \emph{Cayley graph} for $G$ with respect to $S$, denoted by $\Gamma(G,S)$ is a graph with the vertex set $G$ and the edge set given by $$\{ (g,gs) : g \in G , s \in S \}.$$
\end{defn}

 Then it is an easy fact that $\Gamma(G,S)$ is connected, and locally finite whenever $S$ is a finite set. If $S$ and $T$ are finite generating sets for a group $G$, 
 then usually $\Gamma(G,S)$ is different from the graph $\Gamma(G,T)$ but they are quasi-isometric when both $S$ and $T$ are finite.
 
 Clearly, a group $G$ acts on the Cayley graph $\Gamma(G,S)$ by $g \cdot x = gx$. With the metric $$d_S(g,h)=|| g^{-1} h ||_S$$ where $||x||_S$ means the metric between $x \in G$ and the identity $e$ in the graph $\Gamma(G,S)$, this action is isometric and cocompact. 
 When we write the word metric, we sometimes use $||x||$ when there is no confusion for the generating set or the generating set is not important.

 Now we recall $\delta$-hyperbolic spaces and related notions.
 
\begin{defn}
 A metric space $X$ is \emph{$\delta$-hyperbolic} if $X$ is geodesic and there exists some $L \geq 0$ such that for any geodesic triangle $(\gamma_1 , \gamma_2 , \gamma_3 )$ in $X$, the geodesic $\gamma_i$ is contained in a $L$-neighborhood of the other two for all $i=1,2,3$. 
In particular, we say that $X$ is a real tree if it is $\delta$-hyperbolic with $L=0$.
\end{defn}

 A $\delta$-hyperbolic space $X$ has a typical and well-studied boundary called the \emph{Gromov boundary} and we denote it by $\partial X$. To define the Gromov boundary we first recall the \emph{Gromov product}. For a triple $x,y,z$ in $X$, define the value defined as $$(x|y)_z := \frac{1}{2} \left ( d(x,z) + d(y,z) - d(x,y) \right) \in \RR$$ is called the Gromov product of three points $x,y,z$.

 Now we define the Gromov boundary using the Gromov product. First, we say that a sequence $[ x_n ]$ in $X$ converges at infinity if $$\lim_{(m,n)\to \infty} (x_m | x_n)_p = \infty $$ for fixed $p \in X$. Let $\mathcal{S}$ be a set of sequences that converge at infinity. We declare $[x_n] \approx [ y_n ]$ if and only if $$\lim_{n\to \infty} (x_n | x_n)_p = \infty.$$
 Then the relation $\approx$ is an equivalence relation. The Gromov boundary of $X$, $\partial X$ is defined by $$\partial X := \mathcal{S} / \approx . $$
 
It is facts that $[ x_n ] \not \approx [ y_n ]$ if and only if $ \sup_{m,n} \ (x_m|y_n)_p < \infty ,$
and the Gromov boundary does not depend on the choice of the basepoint $p$.

 Let us review the terminology of isometry on a $\delta$-hyperbolic space $X$. An isometry $f$ on $X$ is called \emph{elliptic} if the orbit of $f$ is bounded for some, equivalently any, orbit. $f$ is called \emph{parabolic} if $f$ has exactly one fixed point in the boundary $\partial X$. Lastly, we say that $f$ is \emph{loxodromic} if it has exactly two fixed points in the boundary. Recall that $f$ is loxodromic if and only if the map $\ZZ \to X$ given by $n \mapsto f^n \cdot x$ is a quasi-isometric embedding for any basepoint $x$ in $X$.
 We say that two loxodromic elements $f, \ g$ are \emph{independent} if their fixed point sets are disjoint.

 Suppose that a group $G$ acts isometrically on a $\delta$-hyperbolic group $X$. Then there exists a canonical $G$-action on $\partial X$. For $g \in G$ and $[x_n] \in \partial X$, define
 $$ g \cdot [x_n] := [g \cdot x_n].$$
Then this action is well-defined.

\subsection{Asymptotic cones} \label{Pre:Asymptotic}
 Asymptotic cones are the main ingredient of this paper so we briefly recall that definition of asymptotic cones. Since asymptotic cones are special kinds of ultralimits of metric spaces, we start with the definition of ultrafilters which is a useful concept to define ultralimits.
 
 \begin{defn}
An \emph{ultrafilter} $\omega$ on $\NN$ is a non-empty set of subsets in $\NN$ satisfying
\begin{itemize}
	\item $\emptyset \not \in \omega$.
	\item If $A,B \in \omega$, then $A \cap B \in \omega$.
	\item If $A \in \omega$ and $A \subset B \subset \NN$, then $B \in \omega$.
	\item For any $A \subset \NN$, either $A \in \omega$ or $\NN - A \in \omega$.
\end{itemize}
Also, we say that an ultrafilter $\omega$ is \emph{non-principal} if $F \not \in \omega$ for all finite subsets $F$ in $\NN$.
\end{defn}

 Recall that, from the first and second conditions, it cannot happen that $A \in \omega$ and $\NN-A \in \omega$ hold simultaneously.
 Now we will define an ultralimit of sequence in a metric space. Using this, we can define an ultralimit of metric spaces.

\begin{defn}
 Let $\omega$ be a non-principal ultrafilter on $\NN$. 
When $\left \{ x_n \right \}$ is a sequence of points in a metric space $(X,d)$, 
a point $x \in X$ is called the \emph{ultralimit} of $\left \{x_n\right \}$, denoted by $x:=\lim_{\omega}x_n$ if for every $\epsilon >0$, $$\left \{ n \in \NN : d(x_n,x) \leq \epsilon \right \} \in \omega .$$
\end{defn}

 Similarly to the usual limit, if an ultralimit exists, then it is unique but in general, the ultralimit may not exist. 
However, when a metric space $X$ is compact, it is known that the ultralimit always exists.
This implies that the ultralimit of any bounded sequence in $\RR$ always exists.

 Suppose that a sequence $\{ x_n \}$ converges to $x$ in the usual limit sense. Then for any non-principal ultrafilter $\omega$, $\lim_{\omega} x_n = x$. So, we can think of an ultralimit as a generalization of the usual limits, and this explains why we use non-principal ultrafilters. One of the difficulties of an ultralimit is that an ultralimit depends on the choice of an ultrafilter. 
 Now we define ultralimits of metric spaces.

\begin{defn}
Let $\left( X_n , d_n \right)$ be a sequence of metric spaces. Choose $p_n \in X_n$ for each $n \in \NN$. We call the sequence $\{ p_n \}$ the \emph{observation sequence}.
Suppose that $\omega$ is a non-principal ultrafilter on $\NN$.

A sequence $\{x_n\}$ (here $x_n \in X_n$) is \emph{admissible} if the sequence $\left \{ d_n(x_n,p_n) \right \}$ is bounded. 
 Let $\mathcal{A}$ be the set of all admissible sequences. 
 We define $$X_{\infty} := \mathcal{A} / \sim $$ where $\{x_n\} \sim \{y_n\}$ if and only if $\lim_{\omega}d_n( x_n , y_n )=0$. 
Then the space $X_{\infty}$ is a metric space with the metric defined by $$d_{\infty} \left( \{x_n\},\{y_n\} \right) := \lim_{\omega} d_n\left( x_n, y_n \right).$$
 We say that the metric space $(X_{\infty} , d_{\infty})$ is the \emph{ultralimit} of $\left \{ (X_n,d_n) \right \}$ and denoted by $\lim_{\omega}(X_n,d_n,p_n) := (X_{\infty},d_{\infty}).$
\end{defn}

 Clearly, the ultralimit of metric spaces depends on not only $\left \{ (X_n,d_n)\right \}$ but also the choice of the observation sequence $\{p_n\}$ and ultrafilter $\omega$.

\begin{defn}
 Let $(X,d)$ be a metric space, $\omega$ be a non-principal ultrafilter on $\NN$ and $\{p_n\}$ be a sequence in $X$. Suppose that a sequence $\{d_n\}$ is an unbounded non-decreasing sequence of positive real numbers. 
Then the ultralimit $\lim_{\omega}\left(X,\frac{1}{d_n}d , p_n \right)$ is called the \emph{asymptotic cone} and denoted by $\Cone_{\omega} \left( X , d_n, p_n \right)$.
\end{defn}

 Let $X$ be a $\delta$-hyperbolic space. Recall that both elements of the Gromov boundary of $X$ and an asymptotic cone of $X$ can be considered as sequences in $X$. To avoid confusion, we denote an element of the Gromov boundary by $[x_n] \in \partial X$ and an element of an asymptotic cone by $\{ x_n \} \in \Cone_{\omega} \left( X , d_n, p_n \right)$.

 We define an asymptotic cone of $G$ by an asymptotic cone of its Cayley graph.

\begin{defn}
 Let $G$ be a finitely generated group with a finite generating set $S$. 
An \emph{asymptotic cone} of $G$ is $$\Cone_{\omega}(G,d_n,g_n) := \Cone_{\omega} \left( \Gamma(G,S), d_n, g_n \right)$$ for unbounded non-decreasing sequence $\{d_n\}$ and the observation sequence $\{ g_n \}$.
\end{defn}
 
 Since asymptotic cones are quasi-isometry invariant up to bi-Lipschitz, this definition is well-defined and we can omit a finite generating set $S$. Also, we can remove the observation sequence $g_n$ from the notation. 
 Recall that a metric space $X$ is quasi-homogeneous if there exists a homogeneous space $Y \subset X$ and a constant $C>0$ such that for any $x \in X, d(x,Y) < C$. It is known that $\Cone(X,d_n,p_n)$ does not depend on the choice of the observation sequence when $X$ is quasi-homogeneous. 
 Thus any asymptotic cone of a group does not depend on the choice of the observation sequence $\{ g_n \}$ so we can simply write $\Cone_{\omega} (G,d_n).$ 

 However, an asymptotic cone of a group still depends on the choice of a real sequence $d_n$ and ultrafilters. First, S. Thomas and B. Velickovic constructed a group with two distinct (non-homeomorphic) asymptotic cones \cite{MR1734187}. Their group is only finitely generated, not finitely presented. Later, a finitely presented group with two non-homeomorphic asymptotic cones is constructed \cite{MR2310154}. This group has a simply connected space (not a real tree) and a non-simply connected space as an asymptotic cone.
 
However, it is known that any asymptotic cone of hyperbolic groups is unique and it is a real tree. In general, the converse is not true (recall that a group is called \emph{lacunary hyperbolic} if one of its asymptotic cones is a real tree), but when a group $G$ is finitely presented, then $G$ is hyperbolic if and only if one of its asymptotic cones is a real tree. 
Furthermore, A. Sisto completely classified asymptotic cones of hyperbolic groups as follows.
If a real tree $X$ is the asymptotic cone of a group, then it is a point, a line, or the $2^{\aleph_0}$ universal tree. We refer to Corollary 5.9 in \cite{MR3047634} for further details.

 We end this subsection by introducing the natural action of $G$ on its asymptotic cone.
 Let $G$ be a group and consider one of its asymptotic cones $\Cone_{\omega}(G,d_n)$. Note that an element in the asymptotic cone can be considered as an (admissible) sequence of $G$ so let $\{ x_n \} \in \Cone_{\omega}(G,d_n)$. Define a group action by $$ g \cdot \{ x_n \} := \{ g x_n \}.$$
 It is straightforward that the group action is well-defined and isometric.

\subsection{Acylindrically hyperbolic groups}
 In this subsection, we discuss the acylindrically hyperbolic group. Denis Osin first proposed the notion of an acylindrically hyperbolic group \cite{MR3430352}. This can be seen as a generalization of the relatively hyperbolic group. However, it still includes other significant groups in geometric group theory, e.g., the mapping class groups $\Mod(S)$ where $S$ is not an exceptional case and has no boundary cases, and the outer automorphism groups of the rank $n$ free group $\Out(F_n)$. First of all, we introduce the definition of an acylindrically hyperbolic group.
 
\begin{defn}
 Let $G$ be a group and $X$ be a metric space. An isometric action of $G$ on $X$ is called \emph{acylindrical} if for every $\epsilon > 0$, there exist $R,N > 0$ such that for any two points $x,y \in X$ such that $d(x,y) \geq R$, then $$ \left | \{ g \in G : d(x,gx) \leq \epsilon , d(y,gy) \leq \epsilon \} \right | \leq N.$$ 
\end{defn}
 
 Recall that an isometric action of $G$ on a $\delta$-hyperbolic space $X$ is called \emph{elementary} if the limit set of $G$ on the Gromov boundary $\partial X$ contains at most $2$ points.
 
\begin{defn}
 We say that a group $G$ is \emph{acylindrically hyperbolic} if it admits a non-elementary and acylindrical action on some $\delta$-hyperbolic space $X$. 
\end{defn}

 Indeed, there are equivalent definitions for acylindrical hyperbolicity. The following theorem is well-known and can be found in \cite{MR3430352}. 

\begin{thm}[\cite{MR3430352}, Theorem 1.2] \label{Equi_Defs_AH}
 For any group $G$, the following conditions are equivalent.
    \begin{enumerate}
        \item $G$ admits a non-elementary acylindrical action on a $\delta$-hyperbolic space $X_1$, i.e., $G$ is acylindrical hyperbolic.
        \item There exists a generating set $X$ of $G$ such that the corresponding Cayley graph $\Gamma(G,X)$ is $\delta$-hyperbolic, $|\partial \Gamma(G,X)|>2$, and the natural action of $G$ on $\Gamma(G,X)$ is acylindrical.
        \item $G$ contains a proper infinite hyperbolically embedded subgroup.
        \item $G$ is not virtually cyclic and admits an action on a $\delta$-hyperbolic space $X_2$ such that at least one element of $G$ is loxodromic and satisfies the WPD condition.
    \end{enumerate}
\end{thm} 
 
\begin{conv}
 Throughout this article, for a group property (P), we say that $G$ is \emph{virtually (P)} if $G$ is infinite and $G$ has a finite index subgroup satisfying property (P).
 So, $G$ is \emph{virtually cyclic} if $G$ contains a finite index subgroup isomorphic to $\ZZ$. 
 Notice that any virtually (P) group cannot be finite in this paper.
\end{conv}

 We want to point out that $\delta$-hyperbolic spaces $X_1$ and $X_2$ in the aforementioned theorem may be different.
The terminology WPD stands for ``Weakly Properly Discontinuous" and it was first introduced by Bestvina and Fujiwara \cite{MR1914565}.
We briefly recall its definition.

\begin{defn}
 Let $G$ be a group acting on a metric space $(X,d)$. We say that an element $g \in G$ is \emph{WPD} if it satisfies all of the following conditions.
\begin{itemize}
	\item The order of $g$ is infinite.
	\item For some $x \in X$, the map $\ZZ \to X$ given by $n \mapsto g^n \cdot x$ is a quasi-isometrically embedding.
	\item For each $\epsilon >0$ and each $x \in X$, there exists an integer $m>0$ such that
$$ \left| \left \{ a \in G : d(x,a \cdot x) < \epsilon , d(g^m \cdot x , a g^m \cdot x) < \epsilon  \right \} \right| < \infty. $$ 
\end{itemize}
\end{defn}
 By the definition, if $g$ is WPD for some action $G$ on a $\delta$-hyperbolic space $X$, then $g$ should be a loxodromic element.
Also, if $G$ acts acylindrically on $X$, then every loxodromic element is WPD.
Recall that an acylindrically hyperbolic group contains infinitely many independent loxodromic elements \cite{MR3430352}.
For more details about the equivalence definitions, Theorem \ref{Equi_Defs_AH}, we refer to \cite{MR3589159}.
For readers who are interested in WPD elements, we refer to \cite{MR1914565} and \cite{MR4753310}.

 We say that a subgroup $H$ of $G$ is \emph{s-normal} if the intersection $g^{-1}Hg \cap H$ is infinite for every $g \in G$.
 It is proved that the class of acylindrically hyperbolic groups is closed under taking s-normal subgroups \cite[Lemma 7.2]{MR3430352}. 
 By the definition of s-normal subgroups, every infinite normal subgroup is s-normal so we obtain the following lemma.

\begin{lem} \label{infinite_normalsubgroup_acyl}
 Let $G$ be an acylindrically hyperbolic group and $H$ be an infinite normal subgroup of $G$. Then $H$ is also acylindrically hyperbolic.
\end{lem}

We want to point out that the space does not change in Lemma \ref{infinite_normalsubgroup_acyl}. In other words, 
if $G \curvearrowright X$ is non-elementary and acylindrical, then for any infinite normal subgroup $H < G$, the induced action $H \curvearrowright X$ is also non-elementary and acylindrical.

%%%%%%%%%%%%%%%%%%%%%%%%%%%%%%%%%%%%%%%%%%%%%%%%%%%%%%%%%%%%%%%%%%%%%%

\section{On the Subgroups that Appeared in Condition \eqref{Main_Equalities}} \label{sec:4subgroups}
 In this section, we focus on the subgroups that appeared in the equalities in condition \eqref{Main_Equalities}. 
 We will explain these subgroups more precisely and discuss some properties and remarks. Generally, the equalities among these subgroups fail but we will  prove $$ K(G)=\FC(G)=\mathcal{A}(G) $$ when $G$ is acylindrically hyperbolic from already-known results.
\subsection{The unique maximal finite normal subgroup $K(G)$ and the amenable radical $\mathcal{A}(G)$} \label{subsec_KA(G)}
 First, we will discuss the unique maximal finite normal subgroup $K(G)$. We explain the exact meaning of $K(G)$.

 \begin{defn}
 Let $G$ be a group. \emph{The unique maximal finite normal subgroup} is the finite normal subgroup of $G$, containing all finite normal subgroups of $G$.
 \end{defn}
 Note that if a maximal finite normal subgroup exists, then it should be unique. Hence, ``the" unique maximal finite normal subgroup in the definition makes sense.

 Perhaps one of the natural questions is the existence of $K(G)$. Of course, $K(G)$ may not exist even if $G$ is finitely generated. 
 For a concrete example, notice that there is a finitely generated group $G$ whose center $Z(G)$ is isomorphic to the quotient group $\QQ/\ZZ$ due to Abderezak Ould Houcine \cite{MR2278040}. 
 Recall that the center $Z(G)$ is a characteristic subgroup of $G$, and the subgroups $\left <  \frac{1}{p} \right > < \QQ / \ZZ$ are also characteristic subgroups (it follows from the fact that $\left <  \frac{1}{p} \right >$ is the unique subgroup with $p$ elements). Thus we conclude that the subgroups $\left <  \frac{1}{p} \right > $ are finite normal subgroups of $G$. Since $G$ has subgroups $\left <  \frac{1}{p} \right > $ for every prime $p$, this implies that $G$ does not have the unique maximal finite normal subgroup.

 However, it is proved that any acylindrically hyperbolic group contains $K(G)$. We refer to \cite[Theorem 2.24]{MR3589159}.

 The \emph{amenable radical} $\mathcal{A}(G)$ is the unique maximal amenable normal subgroup. For more details about the amenability, we refer to  \cite[Section 18.3]{MR3753580} and \cite[Chapter 9]{MR3729310}.
In contrast to $K(G)$, it is known that every group has the amenable radical.
When $G$ is acylindrically hyperbolic, its amenable radical is finite \cite[Corollary 7.3]{MR3430352} so it directly implies $K(G)=\mathcal{A}(G)$.

 \subsection{Description of the kernel of actions on asymptotic cones}
 Next, we will describe the kernel of group actions on asymptotic cones. Let $G$ be a group and we fix an ultrafilter $\omega$ and sequence $d_n$. The value $d_{\infty}(\{ x_n \} , \{ g x_n \})$ is a metric between $\{ x_n \}$ and $\{ g x_n \}$ in the asymptotic cone. Notice that this value is exactly
 $$ \lim_{\omega} \frac{|| x_n^{-1} g x_n ||_S}{d_n}. $$
 Here, $S$ is a finite generating set for $G$. Since this value measures how far an element $g$ moves $\{ x_n \}$, the kernel can be expressed as follows.
 
 $$ \ker(G \curvearrowright \Cone_{\omega}(G,d_n)) = \left \{ g \in G : \lim_{\omega} \frac{|| x_n^{-1} g x_n ||_S}{d_n}=0 \textnormal{ for all } x_n \textnormal{ with } \lim_{\omega}\frac{x_n}{d_n} < \infty  \right \}. $$
 
\subsection{Finite conjugacy classes subgroup}
 Now, we give the definition of $\FC(G)$. Recall that FC stands for 'Finite Conjugacy classes'.
 
 \begin{defn} \label{def_FC}
 Let $G$ be a group. We define the \emph{finite conjugacy classes subgroup} $\FC(G)$, also called the \emph{$\FC$-center}, by the set of all elements having finitely many conjugacy classes, that is, $$\FC(G) := \{ g \in G : |\{ x^{-1} g x : x \in G \}| < \infty \}.$$
\end{defn}

 Then it is straightforward that $\FC(G)$ is a normal subgroup of $G$.

\subsection{Easily obtained inclusions and equalities}
 Now we will prove easily obtained equalities and some inclusions.
 
\begin{lem} \label{EasyInclusion}
 If a finitely generated group $G$ has the unique maximal finite normal subgroup $K(G)$, then
 $$ K(G) \subset \FC(G) \subset \ker(G \curvearrowright \Cone_{\omega}(G,d_n))$$
 for all ultrafilter $\omega$ and sequence $d_n$.
\end{lem}
\begin{proof}
 The first inclusion follows directly. Now choose $a \in \FC(G)$. Then the set $\left \{ ||g^{-1}ag||_S : g \in G \right \} \subset \RR$ is bounded. Thus, for any unbounded sequence $d_n$ of $\RR$ and any sequence $x_n$ in $G$, we have $$ \lim_{n \to \infty} \frac{x_n^{-1} a x_n}{d_n}=0 $$ so it means that $a$ is an element of the kernel $\ker(G \curvearrowright \Cone_{\omega}(G,d_n))$.
\end{proof}

 However, these subgroups in Lemma \ref{EasyInclusion} may not coincide in general.
 For the first case, just consider $G = \ZZ$. 
 In order to find a counterexample for the second case, consider the three-dimensional Heisenberg group $H$.
 Let $$x= 
\begin{bmatrix}
1 & 1 & 0 \\ 0 & 1 & 0 \\ 0 & 0 & 1
\end{bmatrix}, y=
\begin{bmatrix}
1 & 0 & 0 \\ 0 & 1 & 1 \\ 0 & 0 & 1
\end{bmatrix}$$ be two generators of $H$, and let $z=[x,y]:=x^{-1}y^{-1}xy$. It is known that the commutator subgroup of $H$ is an infinite cyclic group generated by $z$, and $z$ is a distortion element. 
Here, we say that $g \in G$ is a \emph{distortion element} if $g$ has an infinite order and $$\lim_{k \to \infty} \frac{||g^k||}{k}=0$$ for some word metric $|| \cdot ||$ with respect to some finite generating set. The distortedness can be obtained from $[x^l,y^m]=z^{lm}.$

 Now we will give an example of $\FC(G) \subsetneq \ker(G \to \Isom(\Cone_\omega(G,d_n)))$. Put $G=H$, the Heisenberg group, and pick $g = x$, one of its generators. 
 Then $g \not \in \FC(G)$ since $$y^{-m} g y^m = 
\begin{bmatrix}
1 & 1 & m \\ 0 & 1 & 0 \\ 0 & 0 & 1
\end{bmatrix}.$$
Our goal is to show that $g$ is in $\ker(G \to \Isom(\Cone_\omega(G,d_n))).$ First note that any element in $H$ can be expressed as $$
\begin{bmatrix}
1 & a_k & c_k \\ 0 & 1 & b_k \\ 0 & 0 & 1
\end{bmatrix}=y^{b_k}z^{c_k}x^{a_k}.$$
Then we have $$[g,y^{b_k}z^{c_k}x^{a_k}]=\begin{bmatrix}
1 & 0 & b_k \\ 0 & 1 & 0 \\ 0 & 0 & 1
\end{bmatrix}.$$

Now assume $g \not \in \ker(G \to \Isom(\Cone_\omega(G,d_n)))$. Then there exists a sequence $\{ x_n \}$ in $H$ such that $\lim_{\omega} \frac{||x_n||_S}{d_n} = L > 0 $ and $\{ x_n \} \neq g \cdot \{ x_n \}$, equivalently, $$\lim_{\omega} \frac{|| x_n ^{-1} g x_n ||_S}{d_n} = d > 0.$$ 
Since $g$ is fixed and $\lim_{\omega} \frac{d_n}{||x_n||_S}=1/L > 0$, we have $$\lim_{\omega} \frac{|| g^{-1} x_n ^{-1} g x_n ||_S}{||x_n||} = \lim_{\omega} \frac{|| g^{-1} x_n ^{-1} g x_n ||_S}{d_n} \frac{d_n}{||x_n||_S} = \frac{d}{L} > 0$$
so it means that $$ \lim_{\omega} \frac{||[g,x_n]||_S}{||x_n||_S} > 0.$$ Let $x_n = y^{b_n}z^{c_n}x^{a_n}$ so $$[g,x_n]=\begin{bmatrix}
1 & 0 & b_n \\ 0 & 1 & 0 \\ 0 & 0 & 1
\end{bmatrix}=z^{b_n}.$$
Since $z$ is distortion, $\lim_{n \to \infty} \frac{||z^{b_n}||_S}{b_n}=0$ and this implies $\lim_{\omega} \frac{||z^{b_n}||_S}{b_n}=0$.
However, $||x_n||_S \geq b_n$. It is a contradiction. Therefore, $g \in \ker(G \to \Isom(\Cone_\omega(G,d_n))) - \FC(G)$ and it means $$\FC(G) \subsetneq \ker(G \curvearrowright \Cone_\omega(G,d_n)).$$

Also, the following inclusion is known to be satisfied for a countable group. Note that we concentrate only on finitely generated groups so this inclusion holds in our setting. 
\begin{lem} \label{F<A}
 Let $G$ be a countable group. Then we have $$ \FC(G) \subset \mathcal{A}(G).$$
\end{lem}

This inclusion follows immediately from $\AC$-center. This notion was first suggested in \cite{MR4142752} but we use the definition introduced in \cite{MR4403143}.

\begin{defn}
 Let $G$ be a countable group. $\AC$\emph{-center} of $G$, denoted by $\AC (G)$, is defined by $$ \AC (G):= \{ g \in G : \textnormal{the quotient group } G/C_G\left(\left< \left< g \right> \right>\right) \textnormal{ is amenable} \}. $$
  Here, $\left< \left< S \right> \right>$ is the normal closure of $S$ meaning the minimal normal subgroup containing $S$.
\end{defn}

\begin{proof}[Proof of Lemma \ref{F<A}]
 By definition, $\AC$-center contains $\FC(G)$ since any finite group is amenable.
 Due to Lemma 3.1 in \cite{MR4403143}, $\AC (G)$ is an amenable normal subgroup of $G$ so by the definition, $\AC (G)$ is contained in the amenable radical $\mathcal{A}(G)$.
\end{proof}

 Similarly to the previous inclusions, two subgroups $\FC(G)$ and $\mathcal{A}(G)$ may not be the same generally.
 The Heisenberg group $H$ also provides a counterexample. Recall that $g := x \not \in \FC (H)$ but $\mathcal{A}(H)=H$ since $H$ is nilpotent so $\FC(H) \subsetneq \mathcal{A}(H)$.

 The only remaining inclusion relation is $$ \ker(G \curvearrowright \Cone_\omega(G,d_n)) \textnormal{ and } \mathcal{A}(G).$$
In Section \ref{sec:action}, we give an example for $\ker(G \curvearrowright \Cone_\omega(G,d_n)) \subsetneq \mathcal{A}(G)$ (see Lemma \ref{couterexample1_1}). However, we cannot find a group satisfying the reverse inclusion.

\begin{ques}
 For a finitely generated group $G$, does always the inclusion $$ \ker(G \curvearrowright \Cone_\omega(G,d_n)) \subseteq \mathcal{A}(G) $$ hold?
 Or there exists a finitely generated group $G$ such that $$ \ker(G \curvearrowright \Cone_\omega(G,d_n)) \supsetneq \mathcal{A}(G)?$$
\end{ques}

 Recall that the equality among $K(G), \FC(G)$, and $\mathcal{A}(G)$ does not hold in general but they are the same assuming $G$ is acylindrically hyperbolic.

\begin{lem} \label{Lem_KnownResults}
 Let $G$ be an acylindrically hyperbolic group and $X$ be a $\delta$-hyperbolic space on which $G$ acts non-elementarily and acylindrically. 
Then $$\ker(G \curvearrowright \partial X) = K(G) = \FC(G) = \mathcal{A}(G).$$
\end{lem}
\begin{proof}
 As mentioned before, we already have $K(G) = \mathcal{A}(G)$ since the amenable radical of $G$ is finite \cite[Corollary 7.3]{MR3430352}.
To verify $K(G)=\FC(G)$, consider the quotient $G/K(G)$. By the construction, $K\left(G/K(G)\right)=1$ and it is again acylindrically hyperbolic \cite[Lemma 3.9]{MR3368093}. By Theorem 2.35 in \cite{MR3589159}, it follows $$\FC \left( G/K(G) \right)=1.$$ It means $\FC(G) \subset K(G)$ and the opposite inclusion follows directly. Hence $K(G)=\FC(G)$.

 The only remaining part is $\ker(G \curvearrowright \partial X) = K(G)$. Recall that any finite normal subgroup of $G$ acts trivially on $\partial X$ so $\ker(G \curvearrowright \partial X) \supset K(G)$. The reverse inclusion is obtained from Lemma \ref{infinite_normalsubgroup_acyl}. 
If $N$ is an infinite normal subgroup of $G$, then the induced action $N \curvearrowright X$ is also non-elementarily. So, $N$ cannot be the kernel $\ker(G \curvearrowright \partial X)$ and we have $\ker(G \curvearrowright \partial X) \subset K(G)$.
\end{proof}

%%%%%%%%%%%%%%%%%%%%%%%%%%%%%%%%%%%%%%%%%%%%%%%%%%%%%%%%%%%%%%%%%%%%%%

\section{The Kernel of Acylindrically Hyperbolic Groups Acting on Asymptotic Cones} \label{sec:AHG}
 In this section, we will characterize the kernel of group actions on its asymptotic cones. We start this section with the following lemma. We use it to construct a quasi-isometric embedding.

\begin{lem}[\cite{MR3514064}, Lemma 3.2]
 Let $X$ be a $\delta$-hyperbolic space and, for $1 \leq i \leq k$, let $g_i \in \Isom (X)$ be isometries of $X$ such that for some $x_0 \in X$ we have $$d(x_0,g_i \cdot x_0) \geq 2 (g_j^{\pm 1} \cdot x_0| g_l^{\pm 1} \cdot x_0)_{x_0}+18 \delta + 1 $$
 for all $1 \leq i,j,l \leq k$ except when $j=l$ and the exponent on the $g_j$ and $g_l$ are the same. Then the orbit map $$ \left< g_1 ,\dots , g_k \right> \to X $$ given by $g \mapsto g \cdot x_0$ is a quasi-isometric embedding.
\end{lem}

 We slightly modify the statement as follows. The following version is more suitable for our purpose.
 
 \begin{lem} \label{QIembeddingL}
 Let $G$ be a group and suppose $G$ acts isometrically on a $\delta$-hyperbolic space $X$. Let $x,y$ be two independent loxodromic elements.
 Then there exists $M \in \NN$ such that the orbit map $\left< x^M,y^M \right > \to X $ defined by $$g \mapsto g \cdot o $$ is a quasi-isometric embedding for any $o \in X$.
\end{lem}
\begin{proof}
 First, think of the previous lemma with $k=2$. In this case, the required conditions are the following inequalities.
\begin{align*}
d(x_0,g_1 \cdot x_0) \geq 2(g_1^{\pm1} \cdot x_0|g_2^{\pm1} \cdot x_0)_{x_0} + 18 \delta +1 \\
d(x_0,g_1 \cdot x_0) \geq 2(g_1 \cdot x_0|g_1^{-1} \cdot x_0)_{x_0} + 18 \delta +1 \\
d(x_0,g_1 \cdot x_0) \geq 2(g_2 \cdot x_0|g_2^{-1} \cdot x_0)_{x_0} + 18 \delta +1 \\
d(x_0,g_2 \cdot x_0) \geq 2(g_1^{\pm1} \cdot x_0|g_2^{\pm1} \cdot x_0)_{x_0} + 18 \delta +1 \\
d(x_0,g_2 \cdot x_0) \geq 2(g_1 \cdot x_0|g_1^{-1} \cdot x_0)_{x_0} + 18 \delta +1 \\
d(x_0,g_2 \cdot x_0) \geq 2(g_2 \cdot x_0|g_2^{-1} \cdot x_0)_{x_0} + 18 \delta +1
\end{align*}
Here, the inequalities containing $\pm 1$ must hold for any choice of signs $\pm$.
 Put $x = g_1, y = g_2 \in G$ and $p = x_0 \in X$. By the definition of the Gromov product, it suffices to check the following 3 types of inequalities for some $p \in X$ and $M>0$.
\begin{align*} \label{1st}
 \begin{aligned} 
 d(x^{2M} \cdot p , p) & \geq d(x^M \cdot p , p) + 18 \delta +1 \\
 d(y^{2M} \cdot p , p) & \geq d(y^M \cdot p , p) + 18 \delta +1
 \end{aligned} \tag{$\dag$}
 \end{align*}

 \begin{align*} \label{2nd}
 \begin{aligned} 
  d(y^{2M}\cdot p,p) + d(x^M \cdot p ,p) & \geq 2d(y^M \cdot p, p)+18 \delta + 1 , \\
  d(x^{2M}\cdot p,p) + d(y^M \cdot p ,p) & \geq 2d(x^M \cdot p, p)+18 \delta + 1.
 \end{aligned} \tag{$\dag \dag$}
 \end{align*}

 \begin{align*} \label{3rd}
 \begin{aligned}
d( x^{\pm M} y^{\pm M} \cdot p , p) & \geq d(y^{\pm M} \cdot p , p) + 18 \delta + 1 \\
d( y^{\pm M} x^{\pm M} \cdot p , p) & \geq d(x^{\pm M} \cdot p , p) + 18 \delta + 1 
\end{aligned} \tag{$\dag \dag \dag$}
\end{align*}

 The inequalities in \eqref{1st} and \eqref{2nd} follow from the stable translation length. For an isometry $g$ on $X$, the stable translation length is defined by
 $$ \tau(g):= \lim_{k \to \infty} \frac{d(p,g^k \cdot p)}{k}.$$
 It is known that $\tau(g)$ is well-defined. Also, it does not depend on the choice of the basepoint $p \in X$, and 
 when $X$ is $\delta$-hyperbolic, then $g$ is loxodromic if and only if $\tau(g)>0$. For each $m>0$, $\tau(g^m)=m\tau(g)$.

 From these facts, we can obtain the first inequalities. Suppose $ d(x^{2n} \cdot p , p) < d(x^n \cdot p , p) + 18 \delta +1 $ for all $n$. Then we have 
 $$ \frac{d(x^{2n} \cdot p , p)}{n} < \frac{d(x^n \cdot p , p)}{n} + \frac{18 \delta +1}{n}. $$
 By taking $n \to \infty$, we have $\tau(x^2)<\tau(x)$ but $\tau(x)>0$. It is a contradiction. Indeed, this contradiction implies that the inequalities hold for sufficiently large $n$.

Similarly, we can show that all inequalities in \eqref{2nd} are satisfied for sufficiently large $M$.

The inequalities in \eqref{3rd} follow from the assumption that $x,y$ are independent. Since $x,y$ are independent, we get
$$ \lim_{n \to \infty} (x^{\pm n} \cdot p|y^{\pm n} \cdot p)_p < \infty $$ for any choice of signs $\pm$.
 This implies that for all $n$, $$ d(x^{\pm n} \cdot p , p)+d(y^{\pm n} \cdot p , p)-d(x^{\pm n} \cdot p , y^{\pm n} \cdot p) < B $$ for some $M$.
 Since $d(x^{\pm n} \cdot p,p) \to \infty$ as $n \to \infty$, we have $$ \lim_{n \to \infty} d(x^{\pm n} \cdot p , y^{\pm n} \cdot p) - d(y^{\pm n} \cdot p , p) = \infty. $$
 This means that there exists $M_1>0$ such that $$ d( x^{\pm N} y^{\pm N} \cdot p , p) - d(y^{\pm N} \cdot p , p) \geq 18 \delta + 1 $$ 
whenever $N>M_1$.
 Similarly, there exists $M_2>0$ such that $$ d( y^{\pm N} x^{\pm N} \cdot p , p) - d(x^{\pm N} \cdot p , p) \geq 18 \delta + 1 $$ 
whenever $N>M_2$.

Therefore, we show that all inequalities hold for sufficiently large $M$.
 So, the subgroup $\left < x^M, y^M \right>$ satisfies all conditions in Lemma $3.2$ in \cite{MR3514064}. The result now follows.
\end{proof}

 Now we prove the main result.

\begin{thm} \label{AHG:4-equality-CP}
Let $G$ be a finitely generated acylindrically hyperbolic group. Then 
$$ \ker \left ( G \curvearrowright \Cone_{\omega}(G,d_n) \right ) = K(G) $$
 for all ultrafilter $\omega$ and sequence $d_n$.
\end{thm}
\begin{proof}
 By Lemma \ref{EasyInclusion} and Lemma \ref{Lem_KnownResults}, it suffices to show $\ker \left ( G \curvearrowright \Cone_{\omega}(G,d_n) \right ) \subset \ker \left( G \curvearrowright \partial X \right )$ for some $\delta$-hyperbolic space $X$.
 We choose $X$ such that the action $G \curvearrowright X$ is cobounded (\cite[Theorem 1.2]{MR3430352} or \cite[Theorem 1.7]{MR3685605}). In this case, the limit set $L(G)$ is the same as the Gromov boundary $\partial X$.
 
 In order to obtain the inclusion, we will show that $g \not \in \ker \left( G \curvearrowright \partial X \right )$ implies $g \not \in \left ( G \curvearrowright \Cone_{\omega}(G,d_n) \right)$ so choose $g$ in the compliment of $\ker \left( G \curvearrowright \partial X \right )$.
 Since the set of loxodromic fixed points is dense in $L(G)$ (Corollary $7.4.3$ in \cite{MR3558533}) and $L(G)=\partial X$, there exists a loxodromic element $h \in G$ such that $$ [h^n] \neq g^{-1} \cdot [h^n] \textnormal{ and } \Fix(h) \cap \Fix(g^{-1}hg) = \emptyset.$$
 
 Let $x := h$ and $y:=g^{-1}hg$. Then using the ping-pong lemma, we can prove that $H := \left< x^{k} , y^{k} \right >$ is the free group of rank $2$ for sufficiently large $k$. By Lemma \ref{QIembeddingL}, we may assume that the orbit map $\left< x^k, y^k \right> \to X$ is a quasi-isometric embedding. Let $$ z_n := g^{-1} h^{-nk} g h^{nk} = y^{-nk}x^{nk} \in  H.$$
 Note that the length of $z_n$ in $H$ is $2n$ with respect to the generating set $\{ x^k, y^k \}$. We have
 $$ 2n \leq d(o,z_n \cdot o) \leq C ||z_n||_S $$ for $C := \sup_{s \in S} d(o,s \cdot o)$ where $S$ is a finite generating set for $G$. Note that the lower bound is obtained from Lemma 3.2 in \cite{MR3514064}. 
 This implies
 $$ ||z_n||_S \geq \frac{2}{C}n$$ and we thus get
 $$ \lim_{n \to \infty} \frac{|| h^{-nk}gh^{nk} ||_S}{n} = \frac{2}{C} > 0.$$
 
 Now suppose that we are given an unbounded non-decreasing sequence $d_n$ and any ultrafilter $\omega$. Then $\{ h^{\lfloor d_n \rfloor}\}$ is an admissible sequence, that is, $\{ h^{\lfloor d_n \rfloor}\} \in \Cone_{\omega}(G,d_n)$.
 Then $$\lim_{n \to \infty} \frac{|| h^{-k \lfloor d_n \rfloor} g h^{k \lfloor d_n \rfloor} ||_S}{d_n}=\frac{2}{C}>0$$ and it means $\{ h^{\lfloor d_n \rfloor}\} \neq \{ g h^{\lfloor d_n \rfloor}\}$ in the asymptotic cone $\Cone_{\omega}(G,d_n)$ for any ultrafilter $\omega$. 
\end{proof}

Combining with Lemma \ref{Lem_KnownResults}, we can conclude that the following subgroups coincide when $G$ is finitely generated acylindrically hyperbolic.

\begin{cor} \label{AHG:4-equality}
If $G$ is a finitely generated acylindrically hyperbolic group, then 
\begin{align*}
 \begin{aligned} 
\ker \left ( G \curvearrowright \Cone_{\omega}(G,d_n) \right ) = \ker \left( G \curvearrowright \partial X \right ) = K(G) = \FC(G) = \mathcal{A}(G)
\end{aligned}
\end{align*}
for all ultrafilter $\omega$, sequence $d_n$, and $\delta$-hyperbolic space $X$ on which $G$ acts non-elementarily and acylindrically.
\end{cor}

 Recall that for a finitely generated group $G$, an asymptotic cone of $G$ depends on the choice of ultrafilter $\omega$ and non-decreasing sequence $d_n$. 
 Examples of such cases can be found in \cite{MR1734187} and \cite{MR2310154}. 
 In particular, an acylindrically hyperbolic group may have various asymptotic cones and this is guaranteed by the fact that for any finitely presented group $G$, the free product $G * F_n$ is acylindrically hyperbolic for some $n$, due to Osin \cite{MR3403956}.
 We also point out that $G$ has a lot of asymptotic cones for some relatively hyperbolic group $G$, and it is proved that if its peripheral subgroups have a unique asymptotic cone, then it also has the unique asymptotic cone \cite{MR2827204}.

  According to the characterization, we prove that the kernel of the action of an acylindrically hyperbolic group $G$ on its asymptotic cone is invariant even though $G$ may have a lot of asymptotic cones.

\begin{cor} \label{AHG:Kernel_Invariant}
 Let $G$ be a finitely generated acylindrically hyperbolic group. Then the kernel
 $$ \ker \left ( G \curvearrowright \Cone_{\omega}(G,d_n) \right ) $$ is invariant under ultrafilter $\omega$ and sequence $d_n$.
\end{cor}

 Our main result can be applied to inner amenability, von Neumann algebras, and $C^*$-algebras of groups. It is known that many algebraic properties are related to inner amenability and such algebras of groups. For more details, we refer to Theorem 8.14 in \cite{MR3589159} and the references therein.
Applying our main theorem, we can extend Theorem 8.14 in \cite{MR3589159} to obtain the following corollary.

 \begin{cor}
Let $G$ be a finitely generated acylindrically hyperbolic group. Then the following conditions are equivalent.
\begin{enumerate}[label={\textnormal{(\arabic*)}}]
    \item $G$ has no nontrivial finite normal subgroups.
    \item $G$ contains a proper infinite cyclic hyperbolically embedded subgroup.
    \item $G$ is ICC. In other words, $\FC(G)=1$.
    \item $G$ is not inner amenable.
    \item The amenable radical of $G$ is trivial. 
    \item The natural action of $G$ on its asymptotic cone $\Cone_{\omega}(G,d_n)$ is faithful \\ for any ultrafilter $\omega$ and a sequence $d_n$.
    \item The natural action of $G$ on its asymptotic cone $\Cone_{\omega}(G,d_n)$ is faithful \\ for some ultrafilter $\omega$ and a sequence $d_n$.
    \item The induced action of $G$ on the Gromov boundary $\partial X$ is faithful \\ where $X$ is any $\delta$-hyperbolic space on which $G$ acts non-elementarily and acylindrically.
    \item The reduced $C^*$-algebra of $G$ is simple.
    \item The reduced $C^*$-algebra of $G$ has a unique normalized trace.
\end{enumerate}
 \end{cor}
 The equivalence of $(1),(2),(3),(4),(9),$ and $(10)$ is proved in \cite[Theorem 8.14]{MR3589159}. Our contribution in the list is the equivalence of $(1), (6),$ and $(7),$ and the equivalence of $(1),(3),(5),$ and $(8)$ is a consequence of Lemma \ref{Lem_KnownResults}.

%%%%%%%%%%%%%%%%%%%%%%%%%%%%%%%%%%%%%%%%%%%%%%%%%%%%%%%%%%%%%%%%%%%%%%

\section{Applications to Other Spaces at ``Infinity"} \label{sec:App1}
 We devote this section to connecting the natural action on an asymptotic cone and the canonical action on other spaces at ``infinity".
 The main results in this section demonstrate that the kernels of the canonical group action on many spaces at ``infinity" are the same as the kernel of the natural action on asymptotic cones.
 This section covers the limit set of non-elementary convergence group actions, non-trivial Floyd boundary, and various interesting boundaries of a $\CAT(0)$ group. 
 We prove them by using acylindrical hyperbolicity and comparing both kernels with the subgroups in Lemma \ref{Lem_KnownResults}. 
  
\subsection{Non-elementary convergence groups case} \label{sec:CG}
 First, we will consider non-elementary convergence group cases. As we mentioned in the introduction, non-elementary convergence groups are acylindrically hyperbolic \cite{MR3995017}. 
 In this subsection, we will relate the kernel of the action on its asymptotic cone to the kernel of a group acting on its limit set. We will briefly recall the notion of the convergence group.
 
 The convergence group action was first suggested by Gehring and Martin \cite{MR896224}. 
 They considered only actions on the closed $n$-ball or the ($n-1$)-dimensional sphere and later Tukia extended this to actions on general compact Hausdorff spaces \cite{MR1313451}. 
 This can be seen as a generalization of the action of the Kleinian group by Möbius transformations, as well as a generalization of the action of non-elementary hyperbolic groups on their Gromov boundary and the action of non-elementary relatively hyperbolic groups on their Bowditch boundary.
 
 First, we define the convergence group action.
 Our definition adopts the stronger condition, namely, we will consider only group actions on a compact metrizable space.
 
\begin{defn}
 Let $M$ be a compact metrizable space and suppose that a group $G$ acts on $M$ by homeomorphisms.
 This action is called a \emph{convergence action} if, for every infinite distinct sequence of elements $g_n \in G$, 
 there exists a subsequence $g_{n_k}$ and points $x,y \in M$ (not necessarily distinct) satisfying the following.
 
\begin{itemize}
    \item for every open set $U$ containing $y$ and every compact set $K \subset M -\{ x \}$, there exists $N>0$ such that for every $k > N$, $g_{n_k}(K) \subset U$. 
\end{itemize}

 In this case, we also say that $G$ is a convergence group or $G$ acts on $M$ as a convergence group, etc.
\end{defn}

 When $G$ acts on $M$ as a convergence group, there are two subsets of $M$, \emph{the domain of discontinuity} $\Omega(G)$ (also called the \emph{ordinary set}) and the \emph{limit set} $L(G)$. We define the domain of discontinuity, $\Omega(G)$ by the set of all points in $M$ where $G$ acts properly discontinuously and the limit set as the complement, that is, $L(G) := M - \Omega(G)$. Recall that a group action of $G$ on a topological space $M$ is \emph{properly discontinuous} if for every compact subset $K \subset M$, the set
 $\left \{ g \in G : g \cdot K \cap K \neq \emptyset \right \}$ is finite.

 We want to point out that any finitely generated group can be considered as a convergence group with $|L(G)|=1$ (Example 1.2 in \cite{MR1452414}). 
 Also, for any group $G$ and a discrete topological space $M$ with $|M|=2$, the trivial action of $G$ on $M$ is a convergence group action.
To avoid such trivial cases, we only consider non-elementary convergence groups.
 
\begin{defn}
 Let $G$ be a convergence group acting on $M$. If $|L(G)| \geq 3$, then $G$ is said to be \emph{non-elementary}. Otherwise, we say that it is \emph{elementary}.
\end{defn}

The following fact is well-known.

\begin{fact}
 Let $G$ be a group acting on a compact metrizable space $M$ with $|M|>2$, by homeomorphisms and $\Trip(M):=\left \{ (x,y,z) : x,y,z \in M , |\{ x,y,z \}| = 3 \right \}$ be the set of distinct triples in $M$.
 Then $G$ is a convergence group if and only if the action of $G$ on $\Trip(M)$ is properly discontinuous.
\end{fact}

 This fact implies that, if a subgroup $H$ of $G$ fixes distinct three points on $M$, then $H$ should be finite. Thus the kernel of the convergence group action is finite. 
 Conversely, any finite normal subgroup $N$ of $G$ acts trivially on the limit set $L(G)$ when $G$ is non-elementary (see \cite{martin1986generalizations} and \cite{MR1452414}) so it immediately follows that the kernel of the convergence group action coincides with $K(G)$.
Since a non-elementary convergence group is acylindrically hyperbolic \cite{MR3995017},
applying Corollary \ref{AHG:4-equality}, we have the following corollary.

 \begin{cor} \label{CG:4-equality}
 Let $G$ be a finitely generated group. Suppose that $G$ admits a convergence group action $ G \curvearrowright M$ with $|L(G)| > 2$. Then
$$ \ker \left ( G \curvearrowright \Cone_{\omega}(G,d_n) \right ) = \ker \left( G \curvearrowright L(G) \right ) = K(G) = \FC(G) = \mathcal{A}(G) $$
 for all ultrafilter $\omega$ and sequence $d_n$.
\end{cor}
 
%%%%%%%%%%%%%%%%%%%%%%%%%%%%%%%%%%%%%%%%%%%%%%%%%%%%%%%%%%%%%%%%%%%%%%

\subsection{Groups with non-trivial Floyd boundary case} \label{sec:GFB}
 In this subsection, we investigate another well-understood boundary, the Floyd boundary.
This boundary was first suggested by Floyd in \cite{MR568933}. First, we recall the definition. 
We will use the definition from \cite{MR4052190} and \cite{MR3120738}.

\begin{defn}
 Let $G$ be a finitely generated group and $f:\NN \cup \{ 0 \} \to \RR$ be a function, called the \emph{Floyd function}, satisfying the following conditions.
\begin{itemize}
    \item $f(n)>0$ for all $n \in \NN$ and $f(0)=f(1)$.
    \item There exists $K \geq 1$ such that $$ 1 \leq \frac{f(n)}{f(n+1)} \leq K $$ for all $n$.
    \item $f$ is summable, that is, $$ \sum_{n=1}^{\infty} f(n) < \infty.$$
\end{itemize}
Fix a finite generating set $S$ for $G$. Now we construct a new graph $\Gamma_f(G,S)$. 
Combinatorially, $\Gamma_f(G,S)$ is the same as the Cayley graph $\Gamma(G,S)$.
We assign a length of an edge $l$ between two vertices $v,w$ to $f(n)$ where $n$ is the distance between the identity $e$ and the edge $l$ in the Cayley graph. With this new length, we can give a path metric on $\Gamma_f(G,S)$. Namely, for any two vertices $a,b$, we define the Floyd distance by
$$ d_f(a,b) = \inf \left \{ \sum_{i=1}^n f \left( \min \left \{ d(e,p_{i-1}) , d(e,p_{i}) \right \} \right ) \right \} $$ where the infimum is taken over all paths $a=p_0,p_1, \cdots , p_n=b$.
 The metric completion $\overline{G_f(G,S)}$ is called the \emph{Floyd completion}, and the subspace $\partial_F (G,S,f) := \overline{G_f(G,S)} - G_f(G,S) $ is the \emph{Floyd boundary} of $G$.
\end{defn}

 We say that $G$ has a \emph{non-trivial Floyd boundary} if $|\partial_F(G,S,f)|>2$.

\begin{rmk}
 A Floyd boundary depends on the choice of the Floyd function $f$ but we will omit $f$ in the notation when there is no confusion. Similarly, we will omit $S$ so we write simply $\partial_F G$.
\end{rmk}

 The Floyd boundary is related to many well-known boundaries. For example, Gromov showed that if $G$ is hyperbolic, then $\partial_F G$ is homeomorphic to the Gromov boundary \cite{MR919829}, and Gerasimov proved that if $G$ is relatively hyperbolic, then there is a continuous equivariant map from $\partial_F G$ to $\partial_B G$ \cite{MR2989436}. 
 They obtained the results by putting $f(n)=\lambda^n$ for some $0<\lambda<1$ as a Floyd function.
 In Floyd's original paper \cite{MR568933}, he also showed that for given a geometrically finite discrete subgroup of $\Isom(\HH^3)$, there is a continuous equivariant map from $\partial_F G$ to the limit set $L(G)$, with a Floyd function $f(n)=\frac{1}{n^2+1}$.
 
 It is known that the induced action of $G$ on its non-trivial Floyd boundary is a convergence group action, thanks to A. Karlsson \cite{MR2005231}. Combining the fact that the limit set is exactly $\partial_F G$, we have the following.

\begin{cor} \label{FB:4-equality}
 Suppose that a finitely generated group $G$ has non-trivial Floyd boundary $\partial_F G$. Then 
 $$ \ker \left ( G \curvearrowright \Cone_{\omega}(G,d_n) \right ) = \ker \left( G \curvearrowright \partial_F G \right ) = K(G) = \FC(G) = \mathcal{A}(G) $$
 for all ultrafilter $\omega$ and sequence $d_n$. 
Moreover, the kernel $\ker \left( G \curvearrowright \partial_F G \right )$ is invariant under the choice of the Floyd function.
\end{cor}

%%%%%%%%%%%%%%%%%%%%%%%%%%%%%%%%%%%%%%%%%%%%%%%%%%%%%%%%%%%%%%%%%%%%%%

\subsection{CAT(0) groups with rank-one isometries case} \label{sec:CAT0}
 In this subsection, we apply our main theorem to $\CAT (0)$ spaces.
 Namely, we will relate the kernel of the action on some boundaries of $\CAT (0)$ spaces to the kernel of the action on asymptotic cones. 
 $\CAT (0)$ spaces have many classical boundaries including the visual boundary $\partial_v X$ and Morse boundary $\partial_* X$. The definition of the visual boundary is similar to the Gromov boundary of a proper $\delta$-hyperbolic space but it lacks the desirable properties of the Gromov boundary. Of particular note is that this is not a quasi-isometry invariant \cite{MR1746908}.
Furthermore, it is proved by Cashen \cite{MR3550299} that the Morse boundary is also not a quasi-isometry invariant. 
 This obstructs the boundary to relate to a group. Namely, if $G$ acts geometrically on two $\CAT(0)$ spaces $X_1$ and $X_2$, then the visual boundaries $\partial_v X_1$ and $\partial_v X_2$ may be different. 

  On the other hand, the sublinearly Morse boundary is known to be quasi-isometry invariant so the sublinearly Morse boundary of a group is well-defined whenever $G$ acts geometrically on $\CAT(0)$ space \cite{MR4423805}. When $G$ acts isometrically, properly discontinuously, and cocompactly on $X$, we say that the action is \emph{geometric}.

 Suppose that a group $G$ acts geometrically on a proper $\CAT (0)$ space $X$ and $X$ has a non-empty sublinearly Morse boundary. Then $G$ contains a rank-one isometry \cite[Corollary 3.14]{MR4423140}. Also, it is known that if a group $G$ acts properly on a proper $\CAT (0)$ space $X$ and $G$ contains rank-one isometry $g$, then $g$ is contained in a virtually cyclic hyperbolically embedded subgroup (See Definition 1.2, Proposition 3.14 and Theorem 4.7 in \cite{MR3849623}). 
Note that this hyperbolically embedded subgroup may not be a proper subgroup so we can summarize the above as follows.

\begin{prop}[\cite{MR4423140} and \cite{MR3849623}] \label{Sublinearly-Rank-one-Acylindrical}
 Let $G$ be a group acting geometrically on a proper $\CAT (0)$ space $X$.
 If $G$ contains a rank-one isometry and $G$ is not virtually cyclic, then $G$ is acylindrically hyperbolic.
\end{prop}

 Therefore, based on the assumption above, we can directly use Corollary \ref{AHG:4-equality}.  
 Our first application is to investigate the kernel of actions on the visual boundary. 
First of all, we briefly recall elementary $\CAT(0)$ geometry, important isometry on $CAT(0)$ spaces called rank-one isometry, and many boundaries of a $\CAT(0)$ space including the sublinearly Morse boundary.

\begin{defn}
 Suppose that $X$ is a geodesic space and $\triangle(p,q,r)$ is a geodesic triangle in $X$ with three vertices $p,q,r \in X$. 
 Let $\overline{\triangle}$ be the (unique) triangle in the Euclidean plane $\RR^2$ with the same side lengths. We call it the comparison triangle for $\triangle$.

 We say that $\triangle(p,q,r)$ in $X$ satisfies the \emph{$\CAT(0)$ inequality} if for any points $x,y \in \triangle(p,q,r)$ with $x$ and $y$ on the geodesics $[p,q]$ and $[p,r]$, respectively, and 
 if we choose two points $\overline{x}, \overline{y}$ in $\overline{\triangle}$ such that $d_X(p,x)=d_{\RR^2}(\overline{p},\overline{x})$ and $d_X(p,y)=d_{\RR^2}(\overline{p},\overline{y})$, then we have
 $$ d_X(x,y) \leq d_{\RR^2}(\overline{x},\overline{y}).$$
 We say that $X$ is a \emph{$\CAT(0)$ space} if all geodesic triangles in $X$ satisfy the $\CAT(0)$ inequality. 
\end{defn}

 In this subsection, we assume that a $\CAT(0)$ space is proper. Recall that a metric space $X$ is \emph{proper} if closed balls in $X$ are compact. With properness, $\CAT(0)$ space has more properties, in particular, the nearest-point projections. We explain it more precisely.
 Let $l$ be a geodesic line and $x$ be a point in $X$. Then there exists a unique point $p$ so that $p$ on $l$ and $d_X(x,l)=d_X(x,p)$. We call the point $p$ the \emph{nearest-point projection} from $x$ to $g$, and we denote it by $\pi_l(x)$.
 Next, we record the definition of the $\CAT(0)$ group.

\begin{defn}
 A group $G$ is said to be a \emph{$\CAT(0)$ group} if it acts discretely, cocompactly, and isometrically on some $\CAT(0)$ space.
\end{defn}

 By the definition, we allow a $\CAT(0)$ space to be flat so a free abelian group $\ZZ^n$ is a $\CAT(0)$ group. This is an immediate counterexamples for the generalization of the main theorem to general $\CAT(0)$ groups. However, we will prove that the generalization holds if a $\CAT(0)$ group contains a rank-one isometry. Now we introduce a rank-one isometry.

\begin{defn}
 Let $X$ be a $\CAT(0)$ space. 
\begin{itemize}
	\item An isometry $g$ on $X$ is called \emph{hyperbolic} (or \emph{axial}) if there exists a geodesic line $l:\RR \to X$ which is translated non-trivially by $g$, that is, $g(l(t))=l(t+a)$ for some $a$. The geodesic line $l$ is called an \emph{axis} of $g$.
	\item A \emph{flat half-plane} means a totally geodesic embedded isometric copy of a Euclidean half-plane in $X$.
	\item We say that an isometry $g \in \Isom(X)$ is \emph{rank-one} if it is hyperbolic and some (equivalently, every) axis of $g$ does not bound a flat half-plane.
\end{itemize}
\end{defn}

 Now we describe various boundaries of $\CAT(0)$ space $X$. First, we consider the visual boundary. As mentioned before, its construction is similar to the Gromov boundary of a proper $\delta$-hyperbolic space, namely, we use geodesic lines to define the visual boundary. 
 After establishing the visual boundary, we give a topology on the boundary, called the cone topology. Although there is another well-studied topology on the boundary, the Tits metric, we will not use it so we will omit any introductory explanation about the Tits metric.
 Instead, we refer to \cite{MR1744486} for readers who are interested in the metric.

\begin{defn}
 Let $X$ be a complete $\CAT(0)$ space. Two geodesic rays $c, c': [0,\infty) \to X$ are asymptotic if $$d_X\left( c(t),c'(t) \right) < K$$ for some $K$ and all $t \geq 0$.
 Let $\mathcal{R}(X)$ be the set of all geodesic rays and declare two geodesic rays $c, c'$ are equivalent if and only if two geodesics are asymptotic. Then the quotient space $\mathcal{R}(X)/ \sim$ is called the \emph{visual boundary} of $X$. We denote the visual boundary of $X$ by $\partial_v X$.
\end{defn}

 Now suppose that the geodesic ray $c$ and two positive numbers $t,r$ are given. Consider the following subset of $\mathcal{R}(X)$.
$$ B(c,t,r) := \left \{ c' : c'(0)=c(0) , d_X \left( c(t),c'(t) \right ) < r \right \}.$$
Then it follows that $B(c,t,r)$ forms a basis for a topology on $\partial_v X$. We call this topology the \emph{cone topology}.

 Next, we introduce the Morse boundary. One of the weaknesses of the visual boundary is that it is not a quasi-isometry invariant. As a result, there have been many attempts to construct a boundary of $\CAT(0)$ space which is quasi-isometry invariant, for several years. 
 In 2015, Ruth Charney and Harold Sultan suggested a new boundary called the Morse boundary (also known as the contracting boundary) \cite{MR3339446}. In the paper, we will denote it by $\partial_* X$.
 The Morse boundary is the subset of the visual boundary $\partial_v X$, consisting of Morse geodesic rays. Recall that a geodesic $c$ is called a \emph{Morse} if, for any constants $A \geq 1$ and $B \geq 0$, 
 there is a constant $M(A,B)$ only depending on $A,B$, such that for every $(A,B)$-quasi-geodesic $q$ with endpoints on $c$, we have $q$ is in the $M$ neighborhood of $c$.
 Also, we say that a geodesic $c$ is \emph{$D$-contracting} if for any points $x,y \in X$, 
 \begin{center}
 $d_X(x,y) < d_X(x,\pi_c(x))$ implies $d_X(\pi_c(x),\pi_c(y))<D$.
 \end{center}
 Recall that $\pi_c(x)$ is the nearest-point projection.
 We simply say that a geodesic $c$ is contracting if $c$ is $D$- contracting for some $D$.
 It is known that a geodesic $c$ is contracting if and only if it is Morse so we can make a boundary using contracting geodesics. This explains why people use the Morse boundary and the contracting boundary interchangeably.

 Maybe the most natural topology on $\partial_* X$ is a subspace topology but they suggested another topology. They use a direct limit to topologize. For more details, we refer to their paper \cite{MR3339446}.
 
 From the direction to find the quasi-isometry invariant boundary, Yulan Qing and Kasra Rafi generalize the notion of the Morse geodesic and construct a new boundary called the sublinearly Morse boundary \cite{MR4423805}. 
 As the name suggests, they generalize this by allowing the error term to be sublinear. The precise definition is the following.

\begin{defn}
 A map $\kappa : \RR \to \RR$ is called \emph{sublinear} if $\lim_{t \to \infty} \frac{\kappa(t)}{t}=0$.
\end{defn}
 Now fix a function $\kappa : [0,\infty) \to [1,\infty)$ that is monotone increasing, concave and sublinear.
 For a geodesic line $c$, and a constant $a$, we define the $(\kappa,a)$-neighborhood of $c$ by
 $$\mathcal{N}_{\kappa}(c,a) := \left \{ x \in X : d_X(x,c) \leq a \kappa( d_X(c(0),x) ) \right \}.$$
 We say that two quasi-geodesics $c_1$ and $c_2$ \emph{$\kappa$-fellow travel} each other if
 $$ c_1 \in \mathcal{N}_{\kappa}(c_2,\epsilon) , \ c_2 \in \mathcal{N}_{\kappa}(c_1,\epsilon) $$ for some $\epsilon$.
 Then it is known that $\kappa$-fellow traveling is an equivalence relation.
 Using a sublinear function $\kappa$ and the notion of $\kappa$-fellow traveling, we define a $\kappa$-Morse quasi-geodesic and a $\kappa$-contracting quasi-geodesic as follows.

\begin{defn}[\cite{MR4423805}, Definition 3.6]
 Let $\kappa$ be as before. A quasi-geodesic $c$ is called a \emph{(weakly) $\kappa$-Morse} if there is a function $$m_c : \RR_{+}^2 \to \RR_{+}$$ such that if $\alpha : [s,t] \to X$ is a $(q,Q)$-quasi-geodesic with endpoints on $c$, then
 $$ \alpha([s,t]) \subset \mathcal{N}_{\kappa} (c,m_c(q,Q) ).$$
\end{defn}

We remark that the authors of \cite{MR4423805} also defined the notion of strongly $\kappa$-Morse quasi-geodesic and $\kappa$-contracting quasi-geodesic (See Definition 3.7 and Definition 3.9 in \cite{MR4423805}). But we omit the introduction of these definitions since these are all equivalent \cite[Theorem 3.10]{MR4423805}.
 Now, we define a sublinearly Morse boundary of $X$.

\begin{defn} [\cite{MR4423805}, Definition 4.1]
 Let $X$ be a proper $\CAT(0)$ space and $\kappa$ be as before. The \emph{$\kappa$-Morse boundary} of $X$, denoted by $\partial_{\kappa} X$, is the set of $\kappa$-fellow traveling classes of $\kappa$-Morse quasi-geodesic rays in $X$.
\end{defn}

 We can use Theorem 3.10 in Proposition \cite{MR4423805} to define the $\kappa$-Morse boundary of $X$ as the set of $\kappa$-contracting geodesic rays in $X$, up to $\kappa$-fellow traveling.
 Since distinct geodesic rays do not $\kappa$-fellow travel each other \cite[Lemma 3.5]{MR4423805}, it means $$\partial_{\kappa} X \subset \partial_v X.$$

 In order to relate the kernel $\ker \left ( G \curvearrowright \Cone_{\omega}(G,d_n) \right )$ to other boundaries via the kernel of the actions,
we use the following remarkable relation between $\CAT(0)$ space and $\delta$-hyperbolic space.  % previous version \cite{petyt2022hyperbolic}
In \cite{MR4753310}, for given $\CAT(0)$ space $X$, they define a new metric $D$ such that new metric space $(X,D)$ is $\delta$-hyperbolic.
In addition, they proved a lot of interesting results, for instance, $g$ is rank-one for the action on $\CAT(0)$ space if and only if g acts loxodromically 
on $X_D$ \cite[Theorem C]{MR4753310}.
The space $X_D$ is called the \emph{curtain model} of $\CAT(0)$ space $X$ and the construction is quite technical so we refer to \cite{MR4753310} for readers who are interested in the construction or astonishing results therein.
 To obtain the main result of this subsection, we will use the following lemma.

\begin{lem} \label{X_D-kernel}
 Let $G$ be a finitely generated group $G$ acts geometrically on proper $\CAT(0)$ space $X$. If $G$ contains a rank-one isometry and $G$ is not virtually cyclic, then $$ \ker \left( G \curvearrowright \partial X_D \right ) $$ is finite.
\end{lem}
\begin{proof}
 By Proposition \ref{Sublinearly-Rank-one-Acylindrical}, $G$ is acylindrically hyperbolic.
Let $N$ be arbitrary infinite normal subgroup of $G$. We will show that $N$ cannot be 
$$ \ker \left( G \curvearrowright \partial X_D \right ) $$ so the lemma follows.
By Lemma \ref{infinite_normalsubgroup_acyl}, $N$ is also acylindrically hyperbolic.
 Then by Theorem \ref{Equi_Defs_AH}, there exists a $\delta$-hyperbolic space $Y$ such that $N$ contains at least one WPD (loxodromic) element for the action on $Y$.

 It is proved that if $G$ is $\CAT(0)$ group acting on a $\delta$-hyperbolic space $H$ and $g \in G$ is a WPD element for the action on $H$, then $g$ is also WPD for the action on $X_D$ \cite[Theorem F]{MR4753310}. From this result, $N$ contains a WPD element for the action on $X_D$.
Recall that a WPD element is loxodromic.
To conclude $N \neq \ker \left( G \curvearrowright \partial X_D \right )$, we need to show $|\partial X_D| \geq 3$.

Note that, in our setting, $|\partial_v X|=\infty$ due to \cite[Lemma 16]{MR2545250}.
Also, $\partial X_D$ is known as a dense subset of $\partial_v X$ \cite[Theorem L]{MR4753310}.
Therefore, $\partial X_D$ is infinite and the lemma follows.
\end{proof} 

 We are now ready to characterize the kernel of $\CAT(0)$ groups acting on their asymptotic cones. First, we will show that the kernel is the same as the kernel of the natural action on the visual boundary.

\begin{cor} \label{CAT:4-equality}
 Suppose that a finitely generated group $G$ acts geometrically on proper $\CAT(0)$ space $X$. If $G$ contains a rank-one isometry and $G$ is not virtually cyclic, then we have 
  $$ \ker \left ( G \curvearrowright \Cone_{\omega}(G,d_n) \right ) = \ker \left( G \curvearrowright \partial_{v} X \right ) = \FC(G) = K(G) = \mathcal{A}(G) $$
 for all ultrafilter $\omega$ and sequence $d_n$.
\end{cor}
\begin{proof}
 By Proposition \ref{Sublinearly-Rank-one-Acylindrical}, $G$ is acylindrically hyperbolic so we already have
 $$ \ker \left ( G \curvearrowright \Cone_{\omega}(G,d_n) \right ) = \FC(G) = K(G) = \mathcal{A}(G). $$
 To conclude the result, it suffices to show that $$\FC (G) \subset \ker \left( G \curvearrowright \partial_{v} X \right ) \textnormal{ and } \ker \left( G \curvearrowright \partial_{v} X \right ) \subset \ker \left( G \curvearrowright \partial X_D \right ) \subset K(G)$$
 where $X_D$ is the curtain model of $X$.

 We show the first inclusion. 
 Choose $g \in \FC (G)$ and $[\gamma] \in \partial_{v} X$. Since $[\gamma]$ is a geodesic ray and $G$ acts cocompactly on $X$, there exists a compact set $K$ and a sequence $g_n$ in $G$ such that $$ \gamma \subset \bigcup_{k=1}^\infty g_k \cdot K.$$
 We assume that we choose $g_k$ as minimal, that is, $g_k \cdot K \cap \gamma \neq \emptyset$ for all $k$. Also, we can assume that $\gamma(Ck) \in g_k \cdot K$ for some $C>0$.
 Then $g \cdot \gamma(Ck) \in g g_k \cdot K$ so for any $k$, we have
 $$ d_X(\gamma(Ck) , g \cdot \gamma(Ck)) \leq \diam \left( g_k \cdot K \cup gg_k \cdot K \right) = \diam \left( K \cup g_k^{-1}gg_k \cdot K \right) < ML$$
 where $$\diam(K) < M \textnormal{ and } L := \max \{ ||x^{-1}gx|| : x \in G \}.$$ Recall that $L$ exists since $g \in \FC(G)$.
 Thus $[\gamma] = g \cdot [\gamma]$ and $ g \in \ker \left( G \curvearrowright \partial_{v} X \right )$.

 Now consider the second inclusion. By Lemma \ref{X_D-kernel},  the only remaining thing is to prove
 $$\ker \left( G \curvearrowright \partial_{v} X \right ) \subset \ker \left( G \curvearrowright \partial X_D \right ).$$
Choose $g \not \in \ker \left( G \curvearrowright \partial X_D \right)$, then there exists $[\gamma] \in \partial X_D$ such that $[\gamma] \neq g \cdot [\gamma]$. From \cite[Theorem L]{MR4753310}, there exists an embedding $f:\partial X_D \to \partial_v X$ and 
it directly implies $g \not \in \ker \left( G \curvearrowright \partial_{v} X \right )$.
 We complete the proof. 
\end{proof}

 Now we show that the kernel of the action on the other boundaries, sublinearly Morse boundary and Morse boundary are the same as the kernel of the action on the visual boundary.
 
\begin{cor} \label{CAT:boundaries}
Suppose that a finitely generated group $G$ acts geometrically on proper $\CAT(0)$ space $X$. Assume that $G$ contains a rank-one isometry and $G$ is not virtually cyclic.
 The following three kernels are the same for any sublinear function $\kappa$.
 $$ \ker \left( G \curvearrowright \partial_{v} X \right ) = 
 \ker \left( G \curvearrowright \partial_{\kappa} X \right ) = 
 \ker \left( G \curvearrowright \partial_{*} X \right ).$$
\end{cor}
\begin{proof}
 First, the inclusions 
 \begin{align*} \label{inclusions}
 \partial_* X \subset \partial_{\kappa} X \subset \partial_v X
 \tag{$\star$}
\end{align*}
 follow from Proposition 4.10 in \cite{MR4423805} and Lemma 2.13 in \cite{qinggeometry}.

 Now consider the Morse boundary $\partial_* X$ with the cone topology.
Murray proved that the Morse boundary is dense in the visual boundary with respect
to the cone topology whenever $G$ is not virtually cyclic and $X$ admits a proper cocompact group action \cite{MR3947271}.
Our setting satisfies all the conditions so the Morse boundary is a dense subset in the visual boundary.
From the density, we obtain that the two kernels $$ \ker \left( G \curvearrowright \partial_{v} X \right ), \ 
 \ker \left( G \curvearrowright \partial_{*} X \right )$$ are the same.

 Then the kernel $\ker \left( G \curvearrowright \partial_{\kappa} X \right )$ should be the same as the other two kernels. The result follows from the inclusion \eqref{inclusions}.
\end{proof}

Recall that the sublinearly Morse boundary depends on the choice of a sublinear function $\kappa$ but Corollary \ref{CAT:4-equality} tells us that the kernel is invariant under the choice of $\kappa$.

\begin{cor}
Suppose that a finitely generated group $G$ acts geometrically on proper $\CAT(0)$ space $X$. Also, assume that $G$ contains a rank-one isometry and that $G$ is not virtually cyclic.
Then the kernel $\ker \left( G \curvearrowright \partial_{\kappa} X \right )$ is invariant under the choice of a sublinear function $\kappa$.
\end{cor}

\begin{rmk}
 One may be wondering if the following holds: if a $\CAT(0)$ group $G$ satisfies
\begin{align*}
 \ker \left ( G \curvearrowright \Cone_{\omega}(G,d_n) \right ) &= \ker \left( G \curvearrowright \partial_{v} X \right ) = K(G) = \FC(G) =  \mathcal{A} (G), \\
 &= \ker \left( G \curvearrowright \partial_{\kappa} X \right )  \\ 
 &= \ker \left( G \curvearrowright \partial_{*} X \right )
\end{align*}
then $G$ contains a rank-one isometry. It is false and counterexamples are easily established by using the direct product.
In particular, let $G = F_2 \times F_2$, then the assumption satisfies (all these subgroups are trivial) but $G$ does not contain a rank-one isometry.
\end{rmk}

 As an application, we deduce the result for Coxeter groups. 
 It is worth noting that all Coxeter groups are $\CAT(0)$ groups due to Davis complex \cite{MR2360474}. 
 For any finitely generated Coxeter group $W$, we can always decompose it as 
 $$W = W_F \times W_A \times W_{L_1} \times \cdots \times W_{L_n}$$ where $W_F$ is a finite Coxeter group, $W_A$ is a virtually abelian group and $W_{L_i}$ are irreducible non-spherical and non-affine Coxeter groups. In the decomposition, possibly each component may be trivial (that is, we allow it to be $W_A=1$ or $W_{L_1}=1$ with $n=1$).
 Furthermore, it is known which Coxeter groups contain a rank-one isometry so we can prove the following Coxeter group version application.

\begin{cor} \label{Application:Coxeter}
 Let $W$ be a finitely generated Coxeter group and $\Sigma$ be the Davis complex of $W$. Decompose $W$ as $$W=W_F \times W_A \times W_{L_1} \times \cdots \times W_{L_n}.$$
 If $W_A$ is trivial and $n=1$, then the following subgroups are all the same for any sublinear function $\kappa$.
\begin{align*}
 \ker \left ( W \curvearrowright \Cone_{\omega}(W,d_n) \right ) &= \ker \left( W \curvearrowright \partial_{v} \Sigma \right ) = K(W) = \FC(W) = \mathcal{A} (W) = W_F. \\
 &= \ker \left( W \curvearrowright \partial_{\kappa} \Sigma \right )  \\ 
 &= \ker \left( W \curvearrowright \partial_{*} \Sigma \right )
\end{align*}
 Moreover, we can replace $\Sigma$ with any proper $\CAT(0)$ space $X$ on which $W$ acts geometrically. 
\end{cor} 
\begin{proof}
 Recall that $W$ acts geometrically on its Davis complex $\Sigma$. For simplicity, put $W_{L_1}=W_L$ so by assumption, we have $W = W_F \times W_L$.
Then $W_L$ does not contain a finite-index subgroup that splits as a direct product of two infinite subgroups (\cite[Theorem 4.1(2)]{MR2333366} and that paper uses the terminology \emph{strongly indecomposable}). Since $W_F$ is finite, $W$ also does not contain a finite-index subgroup that splits as a direct product of two infinite subgroups.

 From Theorem 1.1 in \cite{MR2918313}, every irreducible Coxeter group either is finite, affine, or contains a rank-one isometry.
 Thus $W_L$ contains a rank-one isometry. Since $W_L$ is not virtually cyclic, $W_L$ is acylindrically hyperbolic and so is $W$ due to Lemma 3.9 in \cite{MR3368093}.
 Since $K(W_L)=1$ \cite[Proposition 4.3, Assertion 2]{MR2333366},
 we have $K(W)=W_F$ hence we complete the proof for the case of $\Sigma$, the Davis complex.

 For a general proper $\CAT(0)$ space $X$, the result follows immediately from the fact that $W \curvearrowright \Sigma$ contains a rank-one isometry if and only if $W \curvearrowright X$ contains a rank-one isometry \cite[Theorem 1.1]{MR2918313}.
\end{proof}

%%%%%%%%%%%%%%%%%%%%%%%%%%%%%%%%%%%%%%%%%%%%%%%%%%%%%%%%%%%%%%%%%%%%%%

\section{The kernel of Paulin's Construction} \label{sec:Paulin}
 In this section, we concentrate on another group action on its asymptotic cone. This action was first suggested by Paulin so nowadays this new action is called \emph{Paulin's construction} \cite{MR1105339}.
One of the advantages of using Paulin's construction is to obtain a group action without a global fixed point. 
Recall that the natural action $G \curvearrowright \Cone_{\omega}(G,d_n)$ has a global fixed point $\{ e \}$ where $e$ is the identity of $G$. 
But Paulin's construction has no global fixed point.
We denote it by $$G \overset{p}{\curvearrowright} \Cone_{\omega}(G) $$
and we remove $d_n$ from the notation since we can not choose the sequence $d_n$ freely.
The main idea is to use infinitely many automorphisms of $G$ so the condition $|\Out(G)|=\infty$ is required.
With this condition, we define a new action of $G$ on an asymptotic cone by
$$ g \mapsto \left( \{ x_n \} \mapsto \{ \phi_n(g) x_n \} \right)$$
with a sequence $d_n$ defined by $$ d_n := \inf_{g \in G} \max_{s \in S} d_S \left( g,\phi_n(s)g \right). $$
Recall that, when $\phi_n(g)=g$ it is the canonical action on an asymptotic cone. 
Paulin proved that if $|\Out(G)|=\infty$, then we can choose infinitely many automorphisms $\phi_n \in \Aut(G)$ such that the new action does not have a global fixed point.
In fact, Paulin's construction has other advantages. For more details, we refer to Paulin's paper \cite{MR1105339}.

 When Paulin designed the new action, he assumed that $G$ was hyperbolic. The reason is probably to construct a group action on a real tree without a global fixed point (more precisely, a small action).
 However, the main idea can be extended to a general group with the assumption $|\Out(G)|=\infty$.
 Indeed, A. Genevois confirmed that for any group $G$ with $|\Out(G)|=\infty$, $G$ can act on $\Cone_{\omega}(G,d_n)$ without a global fixed point \cite[Section 5.3]{genevois2018automorphisms}.
 
 The main purpose of this section is to show that if $G$ is finitely generated acylindrically hyperbolic, then the kernel of Paulin's construction, $ \ker \left ( G \overset{p}{\curvearrowright} \Cone_{\omega}(G) \right )$ is the same as the kernel of the natural action on asymptotic cone, hence all the subgroups that appeared in Corollary \ref{AHG:4-equality}.
 We start with the following elementary observations.
 
\begin{lem} \label{Kernel_Characteristic}
 The kernel $\ker \left ( G \curvearrowright \Cone_{\omega}(G,d_n) \right )$ is a characteristic subgroup.
\end{lem}
\begin{proof}
 Suppose that an automorphism $\phi$ of $G$, a finite generating set $S$ for $G$, and a sequence $d_n$ are given. Then $$\phi(S)=\{ \phi(s) : s \in S \}$$ is also a finite generating set for $G$. Thus there exist $C, \ D > 0 $ satisfying 
 $$ C < \frac{||\phi(g)||_S}{||g||_S} < D$$ for every $g \in G$.
 Let $g \in \ker \left ( G \overset{p}{\curvearrowright} \Cone_{\omega}(G,d_n) \right )$. Then for any sequence $\{ x_k\}$ in $G$ such that $ \frac{||x_k||_S}{d_k} < \infty$, we have $$\lim_{\omega} \frac{||x_k^{-1}gx_k||_S}{d_k}=0.$$
 It implies
 $$\lim_{\omega} \frac{||\phi(x_k)^{-1} \phi(g) \phi(x_k)||_S}{d_k} = \lim_{\omega} \frac{||\phi(x_k^{-1}gx_k)||_S}{d_k} = 0.$$
 Thus we have $\phi(g) \in \ker \left ( G \curvearrowright \Cone_{\omega}(G,d_n) \right )$. We complete the proof.
\end{proof}

\begin{lem} \label{two_points_stabilizer}
 Let $G$ be a group and suppose that $G$ admits acylindrical and non-elementary action on a $\delta$-hyperbolic space $X$. Then for any two independent loxodromic elements $h_1,h_2$, the pointwise stabilizer subgroup $\Stab([h_1^{\pm n}],[h_2^{\pm n}])$ is finite. 
\end{lem}
\begin{proof}
 Let $N:=\Stab([h_1^{\pm n}],[h_2^{\pm n}])$.  
 Then $N$ is a subgroup of a virtually cyclic group $S := \Stab([h_1^{\pm n}]).$ 
 Then there exists a finite index infinite cyclic normal subgroup $M < S$.

 Consider $N \cap M$. This intersection is a finite index normal subgroup of $N$ and either $N \cap M$ is trivial or isomorphic to $\ZZ$. 
 If $N \cap M=1$, then $N$ is finite so to deduce a contradiction, assume that $N \cap M$ is isomorphic to $\ZZ$.
 Let $g$ be a generator for $N \cap M$.
 Since $M$ is a finite index subgroup of $S$, $g^{n_1} \in S$ for some $n_1$.
 Recall that $h_1 \in S$ so $\left< h_1 \right>$ is a finite index infinite cyclic subgroup of $S$. Thus, $\left( g^{n_1} \right)^{n_2} \in \left< h_1 \right>$ for some $n_2$.
 Thus, $g^n=h_1^m$ for some $n,m$, and it means that $g$ is a loxodromic element.

 Hence, $g \in N$ and $g$ is loxodromic. Since any loxodromic element has exactly two fixed points in $\partial X$, $N$ cannot contain a loxodromic element. It is a contradiction.
Therefore, $N$ is a finite subgroup.
\end{proof}

\begin{lem} \label{finite_stabilizer_kernel}
 Let $G$ be a group and suppose that $G$ admits acylindrical and non-elementary action on a $\delta$-hyperbolic space $X$. Then there exist finitely many loxodromic elements $l_1 , \cdots , l_p \in G$ satisfying
  $$ \ker(G \curvearrowright \partial X) = \Stab([l_1^n],\cdots,[l_p^n]). $$
\end{lem}
\begin{proof}
 The inclusion $ \subset $ is trivial.
 By Lemma \ref{two_points_stabilizer}, there exist four points $[h_1^{\pm n}],[h_2^{\pm n}] \in \partial X$ such that their pointwise stabilizer subgroup, say $N$, is finite.
 If $N= \ker(G \curvearrowright \partial X)$, we are done. Otherwise, there exist only finitely many elements $g_1 , \cdots , g_m \in N - \ker(G \curvearrowright \partial X)$.
 Recall that the set of loxodromic fixed points is dense in $L(G)$ (Corollary $7.4.3$ in \cite{MR3558533}) so for each $1 \leq i \leq m$, there exists a loxodromic element $k_i$ such that $[k_i^n] \neq g_i[k_i^n]$ in $\partial X$.
 By the construction of $h_1, h_2$ and $k_1 , \cdots , k_m$, if $g \not \in \ker(G \curvearrowright \partial X)$, then $[l^n] \neq g \cdot [l^n] \in \partial X$ for some $l \in \{ h_1^{\pm},h_2^{\pm},k_1 , \cdots , k_m \}$.
 By taking $\{ l_1 , \cdots , l_p \} := \{ h_1, h_2 , k_1 , \cdots , k_m \}$, we complete the proof. 
\end{proof}

 Using the lemmas, we deduce that if $G$ is acylindrically hyperbolic, then the kernels of these two group actions are the same whenever Paulin's construction is well-defined and we choose suitable representatives.
 
\begin{thm} \label{Kernel_Paulin}
 Let $G$ be a finitely generated acylindrically hyperbolic group satisfying that its outer automorphism group $ \Out(G)$ is infinite.
 For any sequence $ [\phi_1] , \ [\phi_2] , \cdots $ in $\Out(G)$ with $[\phi_i]\neq[\phi_j]$, 
 there exist automorphism representatives $\phi_1 , \ \phi_2 , \cdots \in \Aut(G)$ such that the kernels of two actions on an asymptotic cone are the same, that is, 
 $$ \ker \left ( G \curvearrowright \Cone_{\omega}(G,d_n) \right ) = \ker \left ( G \overset{p}{\curvearrowright} \Cone_{\omega}(G) \right )$$ 
 for any ultrafilter $\omega$ and sequence $d_n$.
\end{thm}
\begin{proof}
 For simplicity, put $N_1 := \ker \left ( G \curvearrowright \Cone_{\omega}(G,d_n) \right )$ and $\ N_2 := \ker \left ( G \overset{p}{\curvearrowright} \Cone_{\omega}(G) \right )$.
First, we assume $$ d_n = \inf_{g \in G} \max_{s \in S} d_S \left( g,\phi_n(s)g \right), $$ in other words, $d_n$ is the sequence in Paulin's construction.
 Choose $g \in N_1$. Then since the kernel $N_1$ is characteristic (Lemma \ref{Kernel_Characteristic}), $\phi(g)$ is in $N_1$ for any $\phi \in \Aut(G)$.
 So, in Paulin's construction, $\phi_n(g) \in N_1$ for all $n$ so we get $ g \in N_2. $ The inclusion $N_1 \subset N_2$ follows immediately since $N_1$ is finite.

 In order to show the opposite inclustion, it suffices to show that $g \not \in N_1$ implies $g \not \in N_2$ so choose $g \not \in N_1$. Then $\phi_i(g) \not \in N_1$. From Corollary \ref{AHG:4-equality}, it means that $\phi_i(g) \not \in \ker(G \curvearrowright \partial X)$.
 Here, $X$ is any $\delta$-hyperbolic space on which $G$ acts acylindrically and non-elementarily so we take $X := \Gamma(G,Y)$ as in Theorem 1.2 and 5.4 in \cite{MR3430352}.

 By Lemma \ref{finite_stabilizer_kernel}, choose finitely many loxodromic elements $l_1, \cdots, l_k \in G$.  
 Then there is $j \in \{ 1,2, \cdots , k \}$ such that $ \left \{ i \in \NN : [l_j ^n] \neq \phi_i (g) \cdot [h_j ^n]  \right \} \in \omega. $
 For simplicity, we will write $l:=l_j$ so $l$ is a loxodromic element and $$ \left \{ i \in \NN : [l ^n] \neq \phi_i (g) \cdot [l ^n]  \right \} \in \omega. $$
 Choose constants $C>0$ and $M \geq 0$ such that $ d_X(e,l ^k) \geq Ck-M $ for all $k > 0$.
 From $ [l ^n] \neq \phi_i (g) \cdot [l ^n] $ for $\omega$-almost $i$, we have
 $$ \lim_{k \to \infty} \left( d_X(\phi_i(g) l^k , e) + d_X(l^k,e) - d_X(\phi_i(g) l^k , l^k)  \right) \not \to \infty.$$
 It implies that $d_X(\phi_i(g) l^k , l^k) \geq d_X(l^k,e) \geq Ck-M$.

Also, due to Theorem 5.4 in \cite{MR3430352}, we obtain $ || l^{-k}\phi_i(g)l^k ||_S \geq ||  l^{-k}\phi_i(g)l^k ||_X:=d_X( l^k,\phi_i(g)l^k)$ for any finite generating set $S$ for $G$. 
Therefore, we get $$ || l^{-k}\phi_i(g)l^k ||_S \geq Ck-M. $$
It means $$\lim_{\omega} \frac{ || l^{-\lfloor d_n \rfloor}\phi_n(g)l^{\lfloor d_n \rfloor} ||_S}{d_n} \geq C.$$
It is obvious that $l^{\lfloor d_n \rfloor}$ is an admissible sequence, that is, $\{ l^{\lfloor d_n \rfloor} \} \in \Cone_{\omega}(G,d_n)$.

Now the only remaining thing to prove is that the ultra limit $\lim_{\omega} \frac{ || l^{-\lfloor d_n \rfloor}\phi_n(g)l^{\lfloor d_n \rfloor} ||_S }{d_n}$ is finite.
To do this, we choose a suitable representative $\phi_n(g)$.
Recall that $$ d_n = \inf_{g \in G} \max_{s \in S} d_S \left( g,\phi_n(s)g \right) $$ so for each $n$, there is $x_n \in G$ such that $ d_n \leq \max_{s \in S} d_S \left( x_n,\phi_n(s)x_n \right) \leq d_n + 1. $ In other words, 
$$ d_n \leq \max_{s \in S} ||x_n^{-1}\phi_n(s)x_n||_S \leq d_n + 1. $$
Taking $\tau_n(s):=x_n^{-1}\phi_n(s)x_n$, we have $d_n \leq \max_{s \in S} ||\tau_n(s)||_S \leq d_n + 1$ and 
$[\phi_n]=[\tau_n]$ in $\Out(G)$.
Since $||\tau_n(s)||_S \leq d_n+1$ for all $s \in S$, we obtain
\begin{align*}
\lim_{\omega} \frac{ || l^{-\lfloor d_n \rfloor}\tau_n(g)l^{\lfloor d_n \rfloor} ||_S }{d_n} &\leq
\lim_{\omega} \frac{ 2 || l^{\lfloor d_n \rfloor} ||_S }{d_n} + \lim_{\omega} \frac{ || \tau_n(g) ||_S }{d_n} \\
&\leq 2L + ||g||_S \\ &< \infty
\end{align*}
so it means $g \not \in N_2$. Thus, we conclude that the two kernel coincide when $ d_n = \inf_{g \in G} \max_{s \in S} d_S \left( g,\phi_n(s)g \right)$.

The general case follows from the fact that the kernel $\ker \left ( G \curvearrowright \Cone_{\omega}(G,d_n) \right )$ invariant under the choice of $d_n$.
\end{proof}

 Note that the condition $|\Out (G)| = \infty$ is essential since for some acylindrically hyperbolic group $G$, $\Out(G)$ is finite. For example, the mapping class group $\Mod (S)$ has a finite outer automorphism group; we refer to \cite{MR970079} and \cite{MR2658420}. 

We end this section by establishing an example that two kernels are different. We first notice that Paulin's construction is also valid for any abelian group $G$ if $G$ satisfies the original conditions, that is, $G$ is finitely generated and $|\Out(G)|=\infty$.
From this observation, we can establish an example such that $$ \ker \left ( G \curvearrowright \Cone_{\omega}(G,d_n) \right ) \neq \ker \left ( G \overset{p}{\curvearrowright} \Cone_{\omega}(G) \right ).$$

For a concrete example, take $G=\ZZ^2$.
Since $G$ is abelian, $\ker \left ( G \curvearrowright \Cone_{\omega}(G,d_n) \right ) = G$ for any $\omega$ and $d_n$.
Recall that $G$ is finitely generated and $\Out(G)=\GL(2,\ZZ)$ so Paulin's construction for $G=\ZZ^2$ is well-defined.
Since Paulin's construction does not have global fixed points, it means $\ker \left ( G \curvearrowright \Cone_{\omega}(G,d_n) \right ) \neq G$.
More precisely, choose an infinite sequence $$M_n := \begin{bmatrix}
1 & n \\ 0 & 1 \end{bmatrix} \in \Out(G).$$
Since the inner automorphism group $\Inn(G)$ is trivial, we can consider $M_n$ as a sequence of automorphisms. 
Let $S:=\{ e_1=(1,0) , e_2=(0,1) \}$ be the canonical generating set for $G$, then $M_n \cdot e_1 = (1,0) , \ M_n \cdot e_2 = (n,1)$.
From Paulin's construction, we have
\begin{align*}
 d_n & := \inf_{g \in G} \max_{s \in S} d_S(g,M_n(s)g) \\
 & = \inf_{g \in G} \max_{s \in S} d_S(e,M_n(s)) \ (\because G \textnormal{ is abelian}) \\
 & = n+1
\end{align*}
 Thus Paulin's construction gives the action $\{ e \}=\{ (0,0) \} \mapsto \{ M_n(e_2) e \} = \{ (n,1) \}$.
 Since $$ \frac{d_S(e,M_n(e_2)e)}{d_n}=1 ,$$ it implies 
 $\{ e \} \neq \{ M_n(e_2) e \}$ in $\Cone_{\omega}(G,d_n)$.
 It means that $e_2$ is not in $\ker \left ( G \overset{p}{\curvearrowright} \Cone_{\omega}(G) \right )$
 so we obtain $$ e_2 \not \in \ker \left ( G \overset{p}{\curvearrowright} \Cone_{\omega}(G) \right ) \neq \ker \left ( G \curvearrowright \Cone_{\omega}(G,d_n) \right ) = G.$$

%%%%%%%%%%%%%%%%%%%%%%%%%%%%%%%%%%%%%%%%%%%%%%%%%%%%%%%%%%%%%%%%%%%%%%

\section{The Kernel of General Group Actions on Their Asymptotic Cones} \label{sec:action}
  We have proved that the kernel should be finite if $G$ is acylindrically hyperbolic.
This section is dedicated to the kernel of general groups acting on their asymptotic cone.
 We study the kernel of various groups and relate the kernel of group actions on its asymptotic cone to the elementariness of the group action.
 It is well-known that if $G$ is hyperbolic, then $G$ is either finite, virtually cyclic, or non-elementary. Also, many properties, including both algebraic and geometric ones, are actually equivalent to elementariness. For example, assuming that $G$ is hyperbolic, the following are equivalent.

\begin{itemize}
    \item $G$ is elementary hyperbolic.
    \item $G$ is amenable.
    \item $G$ is of polynomial growth.
    \item $G$ satisfies a non-trivial law \cite{MR2153979}.
    \item $G$ is boundedly generated \cite{MR1452851}.
\end{itemize}

 The first purpose is to give another equivalent condition for the elementariness.
To do this, we suggest the following definition and consider the virtually nilpontent group case.
\begin{defn} \label{Condition ast}
If a finitely generated group $G$ has $K(G)$ and satisfies
\begin{align*}
\begin{aligned} \label{Main_Equalities}
\ker \left ( G \curvearrowright \Cone_{\omega}(G,d_n) \right ) = \ker \left( G \curvearrowright \partial X \right ) = K(G) = \FC(G) = \mathcal{A}(G)
\end{aligned} \tag{$\ast$}
\end{align*}
for some $\delta$-hyperbolic space $X$, 
then we say that $G$ satisfies conidition \eqref{Main_Equalities}.
\end{defn} 
The following proposition implies that we can use the kernel $ \ker \left ( G \curvearrowright \Cone_{\omega}(G,d_n) \right )$ to classify whether $G$ is a non-elementary hyperbolic group or not.

\begin{prop} \label{virtual-nilpotent}
 Let $G$ be a finitely generated virtually nilpotent group. Then $\FC(G)$ is infinite so the kernel $ \ker \left ( G \curvearrowright \Cone_{\omega}(G,d_n) \right ) $ is also infinite.
\end{prop}
\begin{proof}
 Let $H$ be a nilpotent group with $[G:H] < \infty$. Recall that $H$ is a finitely generated nilpotent group so $H$ is noetherian. Letting $$N:= \bigcap_{g \in G} g^{-1}Hg, $$ we obtain a finite index normal subgroup $N$ which is contained in $H$. Note that $N$ is also a finitely generated nilpotent group.

 Consider the center of $N$, $Z(N)$. It is known that the center of a nilpotent group is a normality-large subgroup, that is, the intersection of the center with any nontrivial normal subgroup is nontrivial. 
 From this fact and every finitely generated nilpotent group is virtually torsion-free (see Proposition $13.75$ and Corollary $13.81$ in \cite{MR3753580}), we can deduce that $Z(N)$ is infinite. 
 More precisely, since $N$ is a finitely generated nilpotent group, $N$ is virtually torsion-free so we choose a finite index torsion-free subgroup $F$ of $N$. 
Let $$S:= \bigcap_{g \in G} g^{-1}Fg$$ be a normal core of $F$, then $S$ is a finite index torsion-free normal subgroup of $N$.
 By normality-large subgroup property, the intersection $$ Z(N) \cap S \neq \{ e \}.$$
 Since $S$ is torsion-free, the intersection must be infinite hence $Z(N)$ is infinite.

 Choose $a \in Z(N)$. Then $nan^{-1}=a$ for all $n \in N$. Since $N$ is a finite index normal subgroup of $G$, we have that the set $ \left \{ gag^{-1}: g \in G \right \}$ is finite. It means $a \in \FC(G)$.
 Therefore, $\FC(G)$ is infinite as $Z(N) \subset \FC(G)$.
Also, it follows that the kernel $ \ker \left ( G \curvearrowright \Cone_{\omega}(G,d_n) \right ) $ is infinite from Lemma \ref{EasyInclusion}.
\end{proof}

Therefore, if $G$ is virtually nilpotent, then condition $\eqref{Main_Equalities}$ cannot be satisfied regardless of the existence of $K(G)$. If it exists, then by the definition, $K(G)$ is a finite subgroup so $K(G) \neq \ker \left ( G \curvearrowright \Cone_{\omega}(G,d_n) \right )$. Thus, in this sense, condition $\eqref{Main_Equalities}$ detects properties of the negative curvature of the group.

 Recall that an infinite hyperbolic group is either virtually cyclic or non-elementary. Similarly, if an infinite group $G$ acts acylindrically on some $\delta$-hyperbolic space $X$, then either $G$ has a bounded orbit, is virtually cyclic, or is non-elementary \cite[Theorem 1.1]{MR3430352}.
 From these facts, we obtain the following results.

\begin{cor} \label{elementary-and-infinity}
Let $G$ be an infinite hyperbolic group. Then $G$ is non-elementary if and only if the kernel $ \ker \left ( G \curvearrowright \Cone_{\omega}(G,d_n) \right )$ is finite.
  
 Furthermore, suppose that finitely generated group $G$ acts acylindrically on some $\delta$-hyperbolic space $X$ and contains a loxodromic element. Then the action of $G$ on $X$ is non-elementary (so $G$ is acylindrically hyperbolic) if and only if $ \ker \left ( G \curvearrowright \Cone_{\omega}(G,d_n) \right )$ is finite.
\end{cor}

 Motivated by Proposition \ref{virtual-nilpotent}, one may wonder whether a solvable group can have the trivial kernel.
In the class of solvable groups, there exists an example satisfying condition $\eqref{Main_Equalities}$ excluding only the amenable radical. One of the examples is the Baumslag-Solitar group $\BS(1,n)$ with $n > 1$. Recall that this group has a group presentation
$$\BS(1,n)= \left< a,t : tat^{-1}=a^n \right>.$$

First, it is well-known that $\BS(1,n)$ is solvable (actually, $\BS(1,n)=\ZZ[\frac{1}{n}] \rtimes \ZZ$) and torsion-free so $K(G)=1$ for $G=\BS(1,n)$. 
Next, we will show $ \ker \left ( G \curvearrowright \Cone_{\omega}(G,d_n) \right )=1$. It directly implies that all three subgroups, $K(G), \ \FC(G)$, and the kernel $\ker \left ( G \curvearrowright \Cone_{\omega}(G,d_n) \right )$ are trivial.

In order to show this, we need an asymptotic metric on $\BS(1,n)$. Since every element $x \in \BS(1,n)$ can be uniquely expressed as $$ x=t^{-l}a^Nt^m $$
where $l,m \geq 0$ and $n \not | \ N$ if $l, m > 0$, the following lemma gives full asymptotic metric information.

\begin{lem}[\cite{MR3384081}, Proposition 2.1] \label{BS-metric}
There exist constants $C_1,C_2,D_1,D_2 > 0$ such that for every element $x=t^{-l}a^Nt^m$ of $\BS(1,n)$, $N \neq 0$, we have
$$ C_1(l+m+\log|N|)-D_1 \leq ||x|| \leq C_2(l+m+\log|N|)+D_2 $$
where $||x||$ is the word metric with respect to the generators $a,t$.
\end{lem}

\begin{lem} \label{couterexample1_1}
 Let $G=\BS(1,n)$. Then the kernel $\ker \left ( G \curvearrowright \Cone_{\omega}(G,d_k) \right )$ is trivial.
\end{lem}
\begin{proof}
Choose a non-trivial element $g \in \BS(1,n)$. Using the normal form, we can write $g=t^{-l}a^Nt^m$ for some $l,m,N$, and let $x_k=t^{-l_k}a^{N_k}t^{m_k}$. Then we have 
 $$x_k^{-1}gx_k = (t^{-m_k}a^{-N_k}t^{l_k})t^{-l}a^Nt^m(t^{-l_k}a^{N_k}t^{m_k}) = t^{-m_k}a^{-N_k}t^{l_k-l}a^Nt^{m-l_k}a^{N_k}t^{m_k}.$$
 Obviously, the rightmost part is not the normal form so we need to compute more. By using a relation $ta=a^n t$, we have
\begin{align*}
t^{-m_k}a^{-N_k}t^{l_k-l}a^Nt^{m-l_k}a^{N_k}t^{m_k} &= t^{-m_k}a^{-N_k}(t^{l_k-l}a^N)t^{m-l_k}a^{N_k}t^{m_k} \\
&= t^{-m_k}a^{-N_k}(a^{nN(l_k-l)}t^{l_k-l})t^{m-l_k}a^{N_k}t^{m_k} \\
&= t^{-m_k}a^{nNl_k-nNl-N_k}t^{m-l}a^{N_k}t^{m_k} \\
&= t^{-m_k}a^{nNl_k-nNl-N_k}(t^{m-l}a^{N_k})t^{m_k} \\
&= t^{-m_k}a^{nNl_k-nNl-N_k}(a^{nN_k(m-l)}t^{m-l})t^{m_k} \\
&= t^{-m_k}a^{nNl_k-nNl-N_k+nN_km-nN_k l}t^{m_k+m-l}.
\end{align*}
The last expression 
\begin{align*} \label{form}
t^{-m_k}a^{nNl_k-nNl-N_k+nN_km-nN_k l}t^{m_k+m-l}
\tag{$\S$}
\end{align*}
looks like the normal form. Note that the expression \eqref{form} can be written as
 $$t^{-m_k}a^{n(Nl_k-Nl+N_km-N_k l)-N_k}t^{m_k+m-l}.$$
 Now take $l_k=1, N_k=1$ and $m_k=\lfloor d_k \rfloor$.
Then we can easily check that expressions $$x_k=t^{-1}at^{ \lfloor d_k \rfloor}, \ \textnormal{ and } x_k^{-1}gx_k=t^{-\lfloor d_k \rfloor}a^{n(N-Nl+m- l)-1}t^{\lfloor d_k \rfloor +m-l}$$ are the normal forms, at least for sufficiently large $d_k$.
Thus by Lemma \ref{BS-metric} we have
$$\lim_{k \to \infty} \frac{||x_k||}{d_k} = \alpha \textnormal{ and } \lim_{k \to \infty} \frac{||x_k^{-1}gx_k||}{d_k} = \lim_{k \to \infty} \frac{|| t^{-\lfloor d_k \rfloor}a^{n(N-Nl+m- l)-1}t^{\lfloor d_k \rfloor +m-l} ||}{d_k} = \beta $$
 where $ C_1 \leq \alpha \leq C_2 $ and $ 2 C_1 \leq \beta \leq 2 C_2$ (here, $C_1, C_2$ are constants in Lemma \ref{BS-metric}).
 It means that $\{ x_k \}$ is admissible, that is, $\{ x_k \} \in \Cone_{\omega}(G,d_k)$ and $$g \not \in \ker \left ( G \curvearrowright \Cone_{\omega}(G,d_k) \right ).$$
 Since $g \neq 1$ is arbitrary, the kernel $\ker \left ( G \curvearrowright \Cone_{\omega}(G,d_k) \right )$ is trivial.
\end{proof}

 Now we concentrate on group actions of $\BS(1,n)$ on $\delta$-hyperbolic spaces. Recall that $\BS(1,n)$ is finitely generated solvable but not virtually nilpotent so
 it has a finite index subgroup such that the subgroup admits a focal action on some $\delta$-hyperbolic space \cite[Proposition 3.7]{balasubramanya2022property}. 
 Indeed, all the cobounded actions of $\BS(1,n)$ on $\delta$-hyperbolic space are known in \cite{abbott2019actions}.
 Among these actions, we consider the action of $\BS(1,n)$ on $\HH^2$ via the following representation
 $$\BS(1,n) \to \PSL(2,\RR)$$ given by $$ a \mapsto \begin{bmatrix} 1 & 1 \\ 0 & 1 \end{bmatrix}, \ t \mapsto \begin{bmatrix}
 \sqrt{n} & 0 \\ 0 & \frac{1}{\sqrt{n}}
 \end{bmatrix}.$$

Then it follows that $\ker \left( G \curvearrowright \partial X \right )$ is trivial letting $G=\BS(1,n)$ and $X=\HH^2$.
Combining with Lemma \ref{couterexample1_1}, we obtain the following result.

\begin{prop}
 There exists a finitely generated solvable group $G$ such that $$ \ker \left ( G \curvearrowright \Cone_{\omega}(G,d_n) \right ) = \ker \left( G \curvearrowright \partial X \right ) = K(G) = \FC(G) = 1 $$
 for any ultrafilter $\omega$ and a sequence $d_n$,
 where $X$ is some $\delta$-hyperbolic space and the action $G \curvearrowright X$ is cobounded.
\end{prop}

 The following example shows that ruling out the condition of actions of a group $G$ on a $\delta$-hyperbolic space $X$, there exists a group satisfying condition $\eqref{Main_Equalities}$ excluding the kernel of the action on $\partial X$.

\begin{prop}
$G=SL(3,\ZZ)$ satisfies the following equalities $$ \ker \left ( G \curvearrowright \Cone_{\omega}(G,d_n) \right ) = K(G) = \FC(G) = \mathcal{A}(G)=1.$$
\end{prop}
\begin{proof}
 Margulis' normal subgroup theorem tells us that for any normal subgroup $N$ of $G = SL(3,\ZZ)$, either $N$ is trivial, or $N$ is of finite index. 
 Obviously, $K(G)=1$ and $\mathcal{A}(G)=1$ since $G$ is not amenable ($G$ contains a rank $2$ free group) and any virtually amenable group is again amenable.

 Since $\FC(G) \subset \ker \left ( G \curvearrowright \Cone_{\omega}(G,d_n) \right )$, in order to complete the proof, it suffices to prove that the kernel $\ker \left ( G \curvearrowright \Cone_{\omega}(G,d_n) \right )$ is not a finite index subgroup of $G$.

 For simplicity, put $N:= \ker \left ( G \curvearrowright \Cone_{\omega}(G,d_n) \right )$ and we shall prove that $N$ cannot be of finite index. For $i \neq j$, let $e_{i,j}$ be the matrix in $G=SL(3,\ZZ)$ such that it has $1$ on the main diagonal and $(i,j)$ component and $0$ otherwise. Then the following statements are well-known.
\begin{itemize}
    \item For distinct $1 \leq i,j,k \leq 3$, $[e_{i,j}^M,e_{j,k}^N]=e_{i,k}^{MN}$.
    \item There exists a constant $C>0$ such that $\frac{1}{C} \log(M) \leq || e_{1,2}^M || \leq C \log(M)$ (See 2.14 in \cite{MR1828742}).
\end{itemize}
Thus, there exists $D>0$ such that $$\frac{1}{D} \log(M) \leq || e_{i,j}^M || \leq D \log(M)$$ for every $i \neq j$. Now put $g=e_{1,2}$ and $x_n=e_{2,3}^n$. Then for $k \in \ZZ-\{ 0 \}$, we have $|| [g^k,x_n ]||=||e_{1,3}^{kn}||$ so we have $$\frac{1}{D} \log(kn) \leq || [g^k,x_n ]||  \leq D \log(kn).$$
However, it also holds that $\frac{1}{D} \log(n) \leq || x_n ||  \leq D \log(n).$ Therefore, we obtain
\begin{align*}
\lim_{n \to \infty} \frac{|| x_n ^{-1} g^{k} x_n ||}{|| x_n ||} = \lim_{n \to \infty} \frac{|| g^{-k} x_n ^{-1} g^{k} x_n ||}{|| x_n ||} 
&= \lim_{n \to \infty} \frac{ || [g^{k} , x_n ] ||}{|| x_n ||} \\ 
&\geq \lim_{n\ \to \infty} \frac{1}{D^2} \frac{\log(k)+\log(n)}{\log(n)} = \frac{1}{D^2} > 0.
\end{align*}
 The above means $g^k \not \in N$ for every $k \in \ZZ$, unless $k=0$. But if $N$ is a finite index subgroup of $G$, then for every $a \in G$, $a^M \in N$ for some $M>0$. It implies that $N$ cannot be a finite index subgroup.
 Now the result $N=1$ follows from Margulis' normal subgroup theorem.
\end{proof}

We note that $SL(3,\ZZ)$ does not admit any interesting actions on a $\delta$-hyperbolic space $X$ (see \cite{MR4094562} and 
\cite{bader2022hyperbolic}). With the terminology Property (NL) which stands for No Loxodromics, in \cite{balasubramanya2022property}, $SL(3,\ZZ)$ satisfies Property (NL). 
Thompson group $T$ is another example. Note that $T$ also satisfies Property (NL) \cite{balasubramanya2022property}.

\begin{prop}
Thompson group $T$ satisfies the following equalities $$ \ker \left ( T \curvearrowright \Cone_{\omega}(T,d_n) \right ) = K(T) = \FC(T) = \mathcal{A}(T)=1.$$
\end{prop}
\begin{proof}
Recall that $T$ is simple \cite[Theorem 5.8]{MR1426438} so in order to show the equalities hold, we just check that these normal subgroups are not whole group $T$. First, the simpleness and infiniteness of $T$ imply $K(T)=1$.
In \cite{MR3630641}, we have $\FC(T)=1$. Since $T$ is not amenable \cite{MR3630641}, we have $\mathcal{A}(T)=1$ so the only remaining to show is that the kernel $\ker \left ( T \curvearrowright \Cone_{\omega}(T,d_n) \right ) $ is trivial.

It is well-known that $T$ has a normal form \cite[Theorem 5.7]{MR1426438} and its word metric property \cite[Theorem 5.1]{MR2452818}. Let $x_0,x_1 \in T$ be two generators for $T$. Then it follows directly from Theorem 5.1 in \cite{MR2452818} that there is $C>0$ such that $$ \frac{n}{C} \leq ||x_0^{-n}|| \leq Cn \textnormal{ and } \frac{2n+1}{C} \leq ||x_0^nx_1x_0^{-n}|| \leq (2n+1)C $$
(recall that $x_0^{-n}$ and $x_0^{n}x_1x_0^{-n}$ are already normal forms). Thus,
$$ 0 < \frac{2n+1}{C^2n} \leq \frac{|| x_0^{n}x_1x_0^{-n} ||}{|| x_0^{-n} ||} \leq \frac{C^2(2n+1)}{n}.$$
It means that $\{x_0^{-n}\} \neq x_1 \cdot \{ x_0^{-n} \}$ in the asymptotic cone $\Cone_{\omega}(T,n)$ so for general sequence $d_n$, replacing $x_0^{-n}$ with $x_0^{-\lfloor d_n \rfloor}$, we can conclude that $x_1$ is not in the kernel $ \ker \left ( T \curvearrowright \Cone_{\omega}(T,d_n) \right ) $.
Since $T$ is simple, the kernel $ \ker \left ( T \curvearrowright \Cone_{\omega}(T,d_n) \right ) $ is trivial.
\end{proof}

 We conclude this section with another interesting example and a question that arises from the example.
 There exists a group $G$ that acts on $\delta$-hyperbolic space and satisfies our whole equalities, condition $\eqref{Main_Equalities}$ but $G$ itself is not acylindrically hyperbolic. The braided Thompson group $BV$ serves as an example. 
 This group was first designed by Brin \cite{MR2364825} and Dehornoy \cite{MR2258261}, independently. We refer to these two articles for more details. In this paper, we use the ``tree-braid-tree" form for describing elements in $BV$. 
 Let $\mathcal{I}$ be the set of triples $(T_1 , \sigma , T_2)$ where $T_1$ and $T_2$ are binary trees with $n$ leaves and $\sigma$ is a braid with $n$ strings. In this paper, we only consider a \emph{binary tree} that is rooted, finite, and planar.
 We say that a vertex $v$ in a binary tree is a \emph{leaf} if its degree is $1$. In other words, $v$ is a leaf if and only if $v$ is connected to only one vertex via an edge.

 We will define an equivalence relation called expansion in \cite{fournier2022braided}. Recall that there exists a surjective group homomorphism $ \pi_n : B_n \to S_n $ from the braid group to the symmetric group, for all $n>1$.
 We call the binary tree with $3$ vertices and $2$ leaves, the caret. 
 Let $(T_1,\sigma, T_2) \in \mathcal{I}$ and suppose that $T_1$ and $T_2$ have $n$ leaves. For $1 \leq k \leq n$, the $k$-th expansion of the triple $(T_1,\sigma, T_2)$ is $(T_1',\sigma',T_2')$ where $T_1'$ is a new binary tree obtained from $T_1$ by adding the caret to $k$-th leaf of $T_1$. 
 Similarly, $T_2'$ is a binary tree obtained from $T_2$ by adding the caret to $\pi_n(k)$-th leaf of $T_2$. 
 $\sigma'$ is a braid in $B_{n+1}$ obtained from $\sigma \in B_n$ by bifurcating the $k$-th string into parallel two strings.
 Now we say that two triple $(T_1,\sigma_1,T_2), (T_1',\sigma_2,T_2')$ in $\mathcal{I}$ are equivalent if one is the $k$-th expansion of the other for some $k$.
 Rigorously, it is not an equivalence relation. It is neither reflexive nor transitive. But by considering the equivalence relation generated by this binary relation, we obtain the equivalence relation and as a set, $BV$ is the set of all equivalence classes. To describe an element of $BV$, we abuse the notation, namely, 
we denote it by a triple $(T_1, \sigma ,T_2)$.

 Now we define a group operation on $BV$. For given two elements $(T_1,\sigma_1,T_2), (T_1',\sigma_2,T_2')$, let $U$ be a binary tree such that $U$ is a finitely many consecutive expansions of both $T_2$ and $T_1'$.
 Then up to the equivalence class, these two elements are expressed as $(\tilde{T_1},\tilde{\sigma_1},U) , (U,\tilde{\sigma_2}, \tilde{T_2'})$. Define
 $$ (T_1,\sigma_1,T_2) \circ (T_1',\sigma_2,T_2') = (\tilde{T_1},\tilde{\sigma_1},U) \circ (U,\tilde{\sigma_2}, \tilde{T_2'}) := (\tilde{T_1} , \tilde{\sigma_1} \tilde{\sigma_2} , \tilde{T_2'}).$$
 Then binary operation $\circ$ is well-defined and $BV$ with the binary operation $\circ$ is a group.
 
Also, it is known that there exists a surjection $\pi : BV \twoheadrightarrow V$. This group homomorphism is induced by the group homomorphisms $\pi_n: B_n \to S_n$ and the kernel is denoted by $bP$ \cite{fournier2022braided}, $PBV$ \cite{MR2384840}, and $P_{br}$ \cite{MR3781416}. However, we denote it by $PB_{\infty} := \ker (\pi)$ since it is a well-known fact that the kernel is isomorphic to the pure braid group with infinite strings $PB_{\infty}$. Finally, we remark that $PB_{\infty}$ is the set of all triples $(T,\rho,T)$ where $\rho$ is a pure braid.

 Now we will establish that $BV$ is a counterexample by proving the following series of lemmas.

\begin{lem} \label{BV-notAH}
 The braided Thompson group $BV$ is not acylindrically hyperbolic.
\end{lem}
\begin{proof}
 Recall that $BV$ contains the pure braid group with infinite strings $PB_{\infty}$ as a normal subgroup \cite[Remark 10.6]{aramayona2021asymptotic} and \cite[Corollary 3.2]{MR2392498}. It suffices to check that $PB_{\infty}$ is not acylindrically hyperbolic.

 Recall that if $G$ is acylindrically hyperbolic and $g \in G$ is loxodromic, then the centralizer $C(g)$ is virtually cyclic. Hence, by the definition of acylindrical hyperbolicity, whenever $G$ is acylindrically hyperbolic, there exists $g \in G$ such that its centralizer $C(g)$ is virtually cyclic. To show that $PB_{\infty}$ is not acylindrically hyperbolic, we will show that for any $x \in PB_{\infty}$, the centralizer $C(x)$ cannot be virtually cyclic.

 Choose $x \in PB_{\infty}$. Then $x \in PB_k$ for some $k \in \NN$. Let $\{ \sigma_n : n \in \NN \}$ be the natural generator for $B_{\infty}$.
 Then $\sigma_{l}^2 \in PB_{\infty}$ and $ \sigma_l^2 \in C(x) $ for all $l > k$. This implies that $C(x)$ is not virtually cyclic. $C(x)$ contains $\ZZ^n$ for all $n \in \NN$ and actually, $C(x)$ contains a subgroup isomorphic to $PB_{\infty}$.
 Since $x$ is arbitrary, there is no $g \in PB_{\infty}$ such that $C(g)$ is virtually cyclic so $PB_{\infty}$ is not acylindrically hyperbolic.
\end{proof}

 The braided Thompson group $BV$ acts on $\delta$-hyperbolic space called the ray graph. 
 The ray graph $R$ was first designed by Danny Calegari on his blog \cite{CalegariBlog}, and obviously, the mapping class group of $\RR^2 - C$, denoted by $M$, acts on the ray graph where $C$ stands for the Cantor set. 
 He suggested this space as the analog of the complex of curves.
 Later J. Bavard (\cite{MR3470720} and English-translated version \cite{bavard2018gromov}) showed that the ray graph has infinite diameter and is $\delta$-hyperbolic. Also, it is proved that $M$ admits two independent loxodromic elements.
 It is known that $BV$ is a subgroup of $M$ \cite{fournier2022braided} so clearly, $BV$ also acts on the ray graph $R$ which is $\delta$-hyperbolic. 

\begin{lem} \label{BV-Gromov_kernel}
 Let $R$ be the ray graph. Then the kernel $\ker \left( BV \curvearrowright \partial R \right )$ is trivial.
\end{lem}
\begin{proof}
 Recall that we can consider $BV$ as a subgroup of the mapping class group of $\RR^2 - C$ so we will show that $\ker(M \curvearrowright \partial R)$ is trivial. The result follows immediately from this.

 Let $x \in M$ be an element fixing $\partial R$ pointwisely. Then by \cite{MR3852444}, $x$ preserves all cliques of high-filling rays. In particular, there exist cliques that contain only one high-filling ray (see subsection 2.7.1 in \cite{MR3852444}).
 Thus, $x$ should fix all such high-filling rays.

 Notice that the set of all such high-filling rays is dense in the circle called the simple circle. The notion of the simple circle is introduced in \cite{MR4266358} and the density property now follows from the fact that $M$ acts minimally on the simple circle. 
 Recall that such cliques are $M$-invariant and the simple circle contains all high-filling rays.

 Therefore, $x$ acts trivially on the simple circle. Since the simple circle also contains all short rays, $x$ acts trivially on the ray graph $R$.
 Since the mapping class group $M$ acts faithfully on $R$, we conclude that $x$ is trivial. Therefore, $ \ker \left( M \curvearrowright \partial R \right ) = 1 $.
\end{proof}

 It is known that $BV$ is finitely presented (\cite{MR2364825} and \cite[Theorem 3.1]{MR2384840}) so its asymptotic cone is well-defined. The following lemma shows that the kernel of the action of $BV$ on asymptotic cones is trivial.

 \begin{lem} \label{BV-Cone_kernel}
 The kernel $\ker \left ( BV \curvearrowright \Cone_{\omega}(BV,d_n) \right )$ is trivial for any ultrafilter $\omega$ and sequence $d_n$.
 \end{lem}
 \begin{proof}
  For simplicity, put $N := \ker \left ( BV \curvearrowright \Cone_{\omega}(BV,d_n) \right )$ and fix an ultrafilter $\omega$ and a sequence $d_n$. 
  Asymptotic word metric on $BV$ is well-known due to \cite[Theorem3.6]{MR2514382}.

  First, recall that $BV$ contains Thompson's group $F$. Every element in $F$ can be represented as two binary tree diagrams with the same number of leaves. By adding the trivial braid between two diagrams, every element in $F$ is also an element in $BV$. For example, two generators $A, B$ for $F$ in \cite{MR1426438} can be expressed as the following tree-braid-tree form. See Figure \ref{A,B_in_F}.

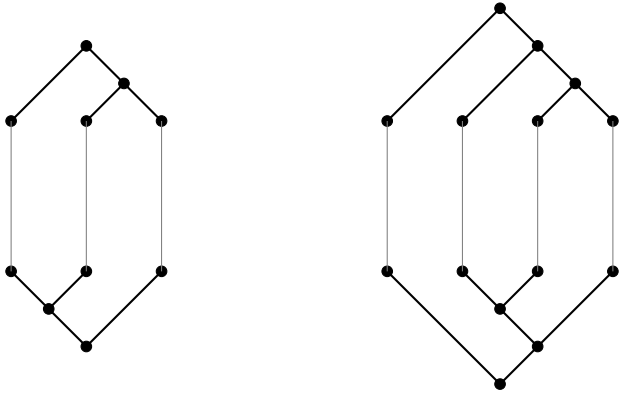
\begin{figure}[h!]
  \begin{center}
\begin{tikzpicture}
% A_1
    \draw[black, thick] (-2,1) -- (-3,2) -- (-4,1);
    \filldraw[black] (-2,1) circle (2pt);
    \filldraw[black] (-3,1) circle (2pt);    
    \filldraw[black] (-4,1) circle (2pt);
    \filldraw[black] (-3,2) circle (2pt);
    \draw[black, thick] (-3,1) -- (-2.5,1.5);
    \filldraw[black] (-2.5,1.5) circle (2pt);
% A_2
    \filldraw[black] (-2,-1) circle (2pt);
    \filldraw[black] (-3,-1) circle (2pt);    
    \filldraw[black] (-4,-1) circle (2pt);
    \filldraw[black] (-3,-2) circle (2pt);
    \draw[black, thick] (-2,-1) -- (-3,-2) -- (-4,-1);
    \filldraw[black] (-3.5,-1.5) circle (2pt);
    \draw[black, thick] (-3,-1) -- (-3.5,-1.5);
% A braid
    \draw[gray, thin] (-2,1) -- (-2,-1); 
    \draw[gray, thin] (-3,1) -- (-3,-1); 
    \draw[gray, thin] (-4,1) -- (-4,-1); 
% B_1
    \draw[black, thick] (1,1) -- (2.5,2.5) -- (4,1);
    \filldraw[black] (1,1) circle (2pt);
    \filldraw[black] (2,1) circle (2pt);    
    \filldraw[black] (3,1) circle (2pt);
    \filldraw[black] (4,1) circle (2pt);
    \filldraw[black] (2.5,2.5) circle (2pt);
    \draw[black, thick] (3,1) -- (3.5,1.5);
    \filldraw[black] (3.5,1.5) circle (2pt);
    \draw[black, thick] (2,1) -- (3,2);
    \filldraw[black] (3,2) circle (2pt);
% B_2
    \draw[black, thick] (1,-1) -- (2.5,-2.5) -- (4,-1);
    \filldraw[black] (1,-1) circle (2pt);
    \filldraw[black] (2,-1) circle (2pt);    
    \filldraw[black] (3,-1) circle (2pt);
    \filldraw[black] (4,-1) circle (2pt);
    \filldraw[black] (2.5,-2.5) circle (2pt);
    \draw[black, thick] (2,-1) -- (3,-2);
    \filldraw[black] (3,-2) circle (2pt);
    \draw[black, thick] (2.5,-1.5) -- (3,-1);
    \filldraw[black] (2.5,-1.5) circle (2pt);   
% B braid
    \draw[gray, thin] (1,1) -- (1,-1);    
    \draw[gray, thin] (2,1) -- (2,-1); 
    \draw[gray, thin] (3,1) -- (3,-1); 
    \draw[gray, thin] (4,1) -- (4,-1); 
\end{tikzpicture}
    \caption{Two generators $A, B$ for $F$ and they are contained in $BV$. The vertical lines in the middle (gray thin lines) represent the trivial braid}
    \label{A,B_in_F}
  \end{center}
\end{figure}

 Consider $X_n := A^{-(n-1)}BA^{n-1}$. Then it is obvious that $X_n$ has $n+3$ leaves and there is no crossing (in the tree-braid-tree form representing $X_n$, the braid part should be trivial). Also, $A^n$ has $n+2$ leaves and there is no crossing in the braid part. By Theorem 3.6 in \cite{MR2514382}, we have
 $$ C_1 n \leq ||A^n|| \leq C_2 n + D_1, \ C_1 n \leq ||X_n|| \leq C_2 n + D_2$$ for some $C_1 , C_2 , D_1 , D_2 > 0$. Thus these two inequalities imply
 $$ C_1 \leq \lim_{\omega} \frac{||X_{\lfloor d_n \rfloor}||}{d_n} \leq C_2.$$
 Since $\{ A^{\lfloor d_n \rfloor} \}$ is admissible (that is, $0 < C_1 \leq \lim_{\omega} \frac{||A^{\lfloor d_n \rfloor}||}{d_n} \leq C_2 < \infty $), we have $ B \not \in N.$

So, the kernel $N$ is a proper subgroup of $BV$. By Corollary 2.8 in \cite{MR3781416}, $N$ is contained in $PB_{\infty}$ which is the kernel of the surjective homomorphism $\pi : BV \twoheadrightarrow V$.

 To complete the proof, we will prove that for any $g \in PB_{\infty}$ with $g \neq 1$, $g \not \in N$. It means $N=1$ so we are done.
 Choose $g \in PB_{\infty}$. Then recall that its tree-braid-tree form is $(T, \rho, T)$ where $T$ is a binary tree diagram with $n$ leaves and $\rho$ is a pure braid with $n$ strings. Since $g \neq 1$, we may assume $\rho \neq 1$. We will use the technique designed in Section 2.2 \cite{fournier2022braided}. To describe this technique, we need to define the \emph{right depth} (\cite[Definition 2.9]{fournier2022braided}), the homomorphism $\chi_1$ (\cite[Definition 2.10]{fournier2022braided}), and so on. 
 We briefly introduce these concepts which are slightly modified for our purpose. More precisely, the only difference is the domain; originally, they defined these for $\widehat{bV}$ but we will define these for a more general case.
 We refer to \cite[Section 2.2]{fournier2022braided} for readers who are interested in the explicit definition of  $\widehat{bV}$ and the original definitions.

 Let $P$ be the subgroup of all elements of the form $(T_1,\beta,T_2)$ where $\beta$ is a pure braid.
 For a given binary tree $T$, the right depth of $T$ is the distance from its rightmost leaf to its root, by giving length $1$ at each edge. Using the right depth, we can define a group homomorphism $$ \chi : P \to \ZZ $$ by $\chi \left ((T_1,\beta,T_2) \right)$ is the right depth of $T_1$ minus the right depth of $T_2$. It is easy to show that it is not well-defined if we consider $\chi$ as a homomorphism on whole group $BV$ but by restriction to $P$, it is well-defined.
 Moreover, $PB_{\infty}$ is contained in the intersection $P \cap \ker(\chi)$.

 Let $K_1$ be the subgroup of $P$ of elements that can be represented by the form $(T_1,\beta, T_2)$ where both $T_1, T_2$ have right depth $1$.
 
\vspace{0.6cm} 

\textbf{Claim 1 : } For any $x \in PB_{\infty}$, there exists $p \geq 0$ such that $A^{-p}xA^p \in K_1.$ \\
\textit{(Proof of Claim 1)} The proof is actually the same as the proof of Lemma 2.12 in \cite{fournier2022braided}. Pick $x \in PB_{\infty}$, then its tree-braid-tree form is $(T,\rho,T)$. If $T$ admits a representative with the right depth $1$, there is nothing to prove.
 Suppose that $T$ has the right depth $r>1$. Then $ A^{-(r-1)}xA^{r-1} \in K_1$. We give an example in Figure \ref{Claim1}.

\begin{figure}[h!]
  \begin{center}
\begin{tikzpicture}
% Left middle
	\draw[red, ultra thick] (-5.3,1) -- (-0.7,1) -- (-0.7,-1) -- (-5.3,-1) -- cycle ;
	\filldraw[black] (-3,0) circle (0pt) node{pure braid $\rho$};
% Left 1
    \filldraw[black] (-1,1) circle (2pt);
    \filldraw[black] (-2,1) circle (2pt);    
    \filldraw[black] (-3,1) circle (2pt);
    \filldraw[black] (-4,1) circle (2pt);
    \filldraw[black] (-5,1) circle (2pt);
    \draw[black, thick] (-1,1) -- (-3,3) -- (-5,1);
    \filldraw[black] (-3,3) circle (2pt);

    \filldraw[black, thick] (-2,1) -- (-1.5,1.5);
    \filldraw[black] (-1.5,1.5) circle (2pt);
    \filldraw[black, thick] (-3,1) -- (-2,2);
    \filldraw[black] (-2,2) circle (2pt);
    \filldraw[black, thick] (-4,1) -- (-2.5,2.5);
    \filldraw[black] (-2.5,2.5) circle (2pt);

    \filldraw[black, thick] (-4.5,1) -- (-4.75,1.25);
    \filldraw[black] (-4.5,1) circle (2pt);
    \filldraw[black] (-4.75,1.25) circle (2pt);
    \filldraw[black, thick] (-1.5,1) -- (-1.75,1.25);
    \filldraw[black] (-1.5,1) circle (2pt);
    \filldraw[black] (-1.75,1.25) circle (2pt);
% Left 2
    \filldraw[black] (-1,-1) circle (2pt);
    \filldraw[black] (-2,-1) circle (2pt);    
    \filldraw[black] (-3,-1) circle (2pt);
    \filldraw[black] (-4,-1) circle (2pt);
    \filldraw[black] (-5,-1) circle (2pt);
    \draw[black, thick] (-1,-1) -- (-3,-3) -- (-5,-1);
    \filldraw[black] (-3,-3) circle (2pt);

    \filldraw[black, thick] (-2,-1) -- (-1.5,-1.5);
    \filldraw[black] (-1.5,-1.5) circle (2pt);
    \filldraw[black, thick] (-3,-1) -- (-2,-2);
    \filldraw[black] (-2,-2) circle (2pt);
    \filldraw[black, thick] (-4,-1) -- (-2.5,-2.5);
    \filldraw[black] (-2.5,-2.5) circle (2pt);

    \filldraw[black, thick] (-4.5,-1) -- (-4.75,-1.25);
    \filldraw[black] (-4.5,-1) circle (2pt);
    \filldraw[black] (-4.75,-1.25) circle (2pt);
    \filldraw[black, thick] (-1.5,-1) -- (-1.75,-1.25);
    \filldraw[black] (-1.5,-1) circle (2pt);
    \filldraw[black] (-1.75,-1.25) circle (2pt);

% Right middle
    \draw[red, ultra thick] (5.3,1) -- (0.7,1) -- (0.7,-1) -- (5.3,-1) -- cycle ;
    \filldraw[black] (3,0) circle (0pt) node{pure braid $\rho$};
% Right 1
    \filldraw[black] (1,1) circle (2pt);
    \filldraw[black] (2,1) circle (2pt);    
    \filldraw[black] (3,1) circle (2pt);
    \filldraw[black] (4,1) circle (2pt);
    \filldraw[black] (5,1) circle (2pt);
    \draw[black, thick] (1,1) -- (3,3) -- (5,1);
    \filldraw[black] (3,3) circle (2pt);

     \draw[black, thick] (2,1) -- (1.5,1.5);
     \filldraw[black] (1.5,1.5) circle (2pt);
     \draw[black, thick] (3,1) -- (2,2);
     \filldraw[black] (2,2) circle (2pt);
     \draw[black, thick] (4,1) -- (2.5,2.5);
     \filldraw[black] (2.5,2.5) circle (2pt);

     \draw[black, thick] (1.5,1) -- (1.25,1.25);
     \filldraw[black] (1.25,1.25) circle (2pt);
     \filldraw[black] (1.5,1) circle (2pt);
    
     \draw[black, thick] (3.5,1) -- (3.75,1.25);
     \filldraw[black] (3.75,1.25) circle (2pt);
     \filldraw[black] (3.5,1) circle (2pt);
% Right 2
    \filldraw[black] (1,-1) circle (2pt);
    \filldraw[black] (2,-1) circle (2pt);    
    \filldraw[black] (3,-1) circle (2pt);
    \filldraw[black] (4,-1) circle (2pt);
    \filldraw[black] (5,-1) circle (2pt);
    \draw[black, thick] (1,-1) -- (3,-3) -- (5,-1);
    \filldraw[black] (3,-3) circle (2pt);

     \draw[black, thick] (2,-1) -- (1.5,-1.5);
     \filldraw[black] (1.5,-1.5) circle (2pt);
     \draw[black, thick] (3,-1) -- (2,-2);
     \filldraw[black] (2,-2) circle (2pt);
     \draw[black, thick] (4,-1) -- (2.5,-2.5);
     \filldraw[black] (2.5,-2.5) circle (2pt);

     \draw[black, thick] (1.5,-1) -- (1.25,-1.25);
     \filldraw[black] (1.25,-1.25) circle (2pt);
     \filldraw[black] (1.5,-1) circle (2pt);
    
     \draw[black, thick] (3.5,-1) -- (3.75,-1.25);
     \filldraw[black] (3.75,-1.25) circle (2pt);
     \filldraw[black] (3.5,-1) circle (2pt);
\end{tikzpicture}
    \caption{An example for Claim 1. The left diagram represents $x=(T,\rho,T)$ where $T$ has the right depth $4$. The right diagram is the tree-braid-tree form of $A^{-3}xA^{3}$ so $A^{-3}xA^{3} \in K_1$.}
    \label{Claim1}
  \end{center}
\end{figure}
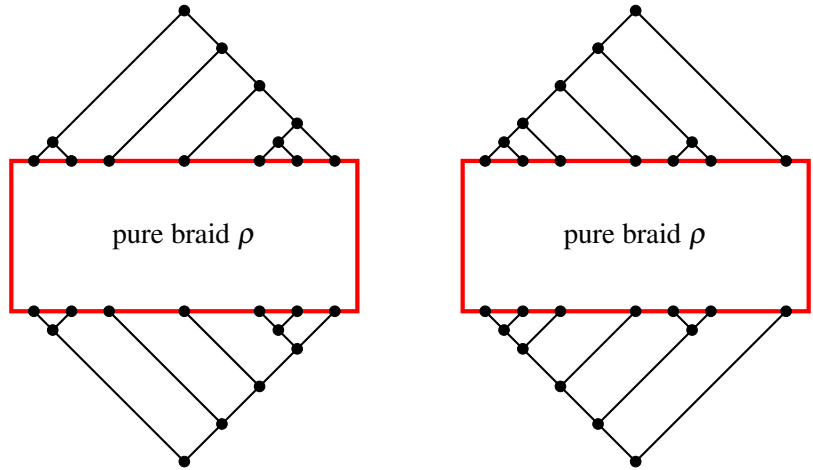
 Since $T$ has the right depth $r$, the number of leaves in $T$ is at least $r+1$. Note that $A^{r-1}$ has the form of $(T_1,1,T_2)$ where the right depth of $T_1, T_2$ are $r$ and $1$, respectively. 
 Since $A^{-(r-1)}$ has the form of $(T_2,1,T_1)$, first we compute $A^{-(r-1)}x$. To do this, we find an expansion $E$ of both $T_1$ and $T$. Since the right depth of $T_1$ and $T$ are both $r$, we can choose an expansion $E$ satisfying the right depth of $E$ is also $r$. Indeed, we just put $E=T$, in other words, $T$ is an extension of $T_1$.
 Thus, using such extension, the tree-braid-tree form of $A^{-(r-1)}x$ is $(T_2',\rho,T)$.
 We want to point out that the right depth of $T_2'$ is $1$.
 Similarly, the tree-braid-tree form of $A^{-(r-1)}xA^{r-1}$ is $(T_2',\rho,T_2')$ but $T_2'$ has right depth $1$. We are done.
 \hfill $\blacksquare$
 
\vspace{0.6cm} 

\textbf{Claim 2 : } For any $x \in PB_{\infty} \cap K_1$ with $l$ leaves, the number of leaves of $A^{-n}xA^n$ is exactly $l+n$. \\
\textit{(Proof of Claim 2)} We will prove that if $x \in PB_{\infty} \cap K_1$ has $l$ leaves, then $A^{-1}xA$ is also in $PB_{\infty} \cap K_1$ and it has $l+1$ leaves. Then the proof is completed by induction.

 First, since $PB_{\infty}$ is a normal subgroup of $BV$, $A^{-1}xA \in PB_{\infty}$.
The facts that $A^{-1}xA \in K_1$ and it has $l+1$ leaves follow from the operation on $BV$.
Choose $x \in PB_{\infty} \cap K_1$. Then its tree-braid-tree form should be as in Figure \ref{Claim2-1}.

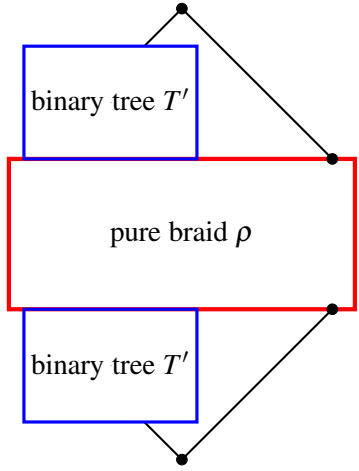
\begin{figure}[h!]
  \begin{center}
\begin{tikzpicture}
% middle
    \draw[red, ultra thick] (5.3,1) -- (0.7,1) -- (0.7,-1) -- (5.3,-1) -- cycle ;
    \filldraw[black] (3,0) circle (0pt) node{pure braid $\rho$};
% 1
    \filldraw[black] (5,1) circle (2pt);
    \draw[black, thick] (2.5,2.5) -- (3,3) -- (5,1);
    \filldraw[black] (3,3) circle (2pt);
    \draw[blue, very thick] (3.2,1) -- (3.2,2.5) -- (0.9,2.5) -- (0.9,1) -- cycle ;
    \filldraw[black] (2.05,1.75) circle (0pt) node{binary tree $T'$};
% 2
    \filldraw[black] (5,-1) circle (2pt);
    \draw[black, thick] (2.5,-2.5) -- (3,-3) -- (5,-1);
    \filldraw[black] (3,-3) circle (2pt);
    \draw[blue, very thick] (3.2,-1) -- (3.2,-2.5) -- (0.9,-2.5) -- (0.9,-1) -- cycle ;
    \filldraw[black] (2.05,-1.75) circle (0pt) node{binary tree $T'$};
\end{tikzpicture}
    \caption{An element $x=(T,\rho,T) \in PB_{\infty} \cap K_1$.}
    \label{Claim2-1}
  \end{center}
\end{figure}
 
 Since $T$ has $l$ leaves, $T'$ has $l-1$ leaves. Here, we assume that $T'$ is reducible, 
 that is, $T'$ is not an extension of any binary tree $T''$ that has less than $l-1$ leaves.
 Then the element $A^{-1}xA$ can be expressed as in Figure \ref{Claim2-2}.

\begin{figure}[h!]
  \begin{center}
\begin{tikzpicture}
% middle
    \draw[red, ultra thick] (6.3,1) -- (0.7,1) -- (0.7,-1) -- (6.3,-1) -- cycle ;
    \filldraw[black] (3.5,0) circle (0pt) node{pure braid $\tilde{\rho}$};
% 1
    \filldraw[black] (5,1) circle (2pt);
    \draw[black, thick] (2.5,2.5) -- (3,3) -- (5,1);
    \filldraw[black] (3,3) circle (2pt);
    \draw[blue, very thick] (3.2,1) -- (3.2,2.5) -- (0.9,2.5) -- (0.9,1) -- cycle ;
    \filldraw[black] (2.05,1.75) circle (0pt) node{binary tree $T'$};
    \draw[black, thick] (3.5,3.5) -- (3,3);
    \filldraw[black] (3.5,3.5) circle (2pt);
    \draw[black, thick] (3.5,3.5) -- (6,1);
    \filldraw[black] (6,1) circle (2pt);
% 2
    \filldraw[black] (5,-1) circle (2pt);
    \draw[black, thick] (2.5,-2.5) -- (3,-3) -- (5,-1);
    \filldraw[black] (3,-3) circle (2pt);
    \draw[blue, very thick] (3.2,-1) -- (3.2,-2.5) -- (0.9,-2.5) -- (0.9,-1) -- cycle ;
    \filldraw[black] (2.05,-1.75) circle (0pt) node{binary tree $T'$};
    \draw[black, thick] (3.5,-3.5) -- (3,-3);
    \filldraw[black] (3.5,-3.5) circle (2pt);
    \draw[black, thick] (3.5,-3.5) -- (6,-1);
    \filldraw[black] (6,-1) circle (2pt);
\end{tikzpicture}
    \caption{Tree-braid-tree form of $A^{-1}xA$.}
    \label{Claim2-2}
  \end{center}
\end{figure}
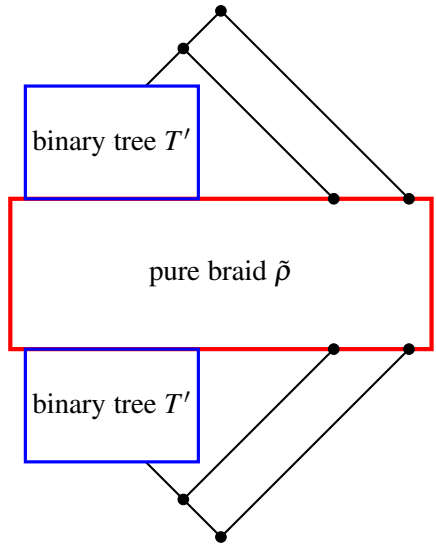

 Here, the pure braid $\tilde{\rho}$ can be constructed from $\rho$ as follows. First, recall that $\rho \in PB_{l}$ is a pure braid. 
 Let $s$ be the $l$-th string (the right-most string in the picture). Now we consider $s$ as two parallel strings, that is, we bifurcate the $l$-th string into parallel two strings. Then the resulting braid is an element of $PB_{l+1}$. This braid is $\tilde{\rho}$.
 From its tree-braid-tree form, $A^{-1}xA \in K_1$ and it has $l+1$ leaves. We complete the proof.
\hfill $\blacksquare$

\vspace{0.6cm} 

 Now we will prove $N=1$. Recall $g \in PB_{\infty}$. Then by Claim 1, $A^{-p}gA^p \in K_1$. Assume that $A^{-p}gA^p$ has $l$ leaves.
 Then by Claim 2, $A^{-n} \left( A^{-p}gA^p \right) A^n$ has $l+n$ leaves. By Theorem 3.6 in \cite{MR2514382}, we obtain
 $$ ||A^{-(n+p)}gA^{n+p}|| \geq C_1(l+n) $$ for some $C_1 >0$.
 Also, recall that $C_2 n \leq ||A^n|| \leq C_3 n + C_4$ for some $C_2, C_3, C_4 > 0$.
Therefore, we get $$ 0 < C_1 \leq \lim_{\omega} \frac{||A^{-(\lfloor d_n \rfloor+p)}gA^{\lfloor d_n \rfloor+p}||}{d_n} \leq 
\lim_{\omega} \frac{||A^{\lfloor d_n \rfloor+p}||}{d_n} \leq 2 C_3 < \infty.$$
Since $\left \{ A^{\lfloor d_n \rfloor+p} \right \}$ is admissible, we show that $g \not \in N$.
Since $g \neq 1$ is arbitrary, we get $N=1$.
 \end{proof}

 \begin{lem} \label{BV-amenalbe_radical}
 The amenable radical of $BV$, $\mathcal{A}(BV)$ is trivial.
 \end{lem}
\begin{proof}
Recall that $PB_{\infty}$ is the kernel of the natural surjection $\pi : BV \twoheadrightarrow V$. 
Due to Corollary 2.8 in \cite{MR3781416}, $\mathcal{A}(G)$ is contained in $PB_{\infty}$ since $\mathcal{A}(G) \neq G$.
Since any amenable normal subgroup of $G$ is also amenable normal subgroup of $PB_{\infty}$, we deduce $\mathcal{A}(G) \subset \mathcal{A}(PB_{\infty})$. Thus we will show that $\mathcal{A}(PB_{\infty})$ is trivial.

To justify this, we need to claim that the pure braid group with $n$ strings, $\mathcal{A}(PB_{n})$ is isomorphic to $\ZZ$ for all $n>1$. We use mathematical induction.
When $n=2$, the braid group $B_2$ is isomorphic to $\ZZ$ and so is $PB_2$. Thus, $\mathcal{A}(PB_{2})=\ZZ$.
Now assume that $\mathcal{A}(PB_{n-1})$ is isomorphic to $\ZZ$.
Recall that there exists a short exact sequence
$$ 1 \to F_{n-1} \to PB_{n} \overset{p}{\to} PB_{n-1} \to 1 $$ where $F_{n-1}$ is the rank $n-1$ free group. It induces
$$ 1 \to F_{n-1} \cap \mathcal{A}(PB_{n}) \to \mathcal{A}(PB_{n}) \overset{p|_{\mathcal{A}(PB_{n})}}{\to} p(\mathcal{A}(PB_{n})) \to 1.$$
 Consider $F_{n-1} \cap \mathcal{A}(PB_{n})$. Recall that $F_{n-1}$ and $\mathcal{A}(PB_{n})$ are both normal in $PB_n$. So, it implies that $F_{n-1} \cap \mathcal{A}(PB_{n})$ is a normal subgroup of both $F_{n-1}$ and $\mathcal{A}(PB_n)$. 
 Thus, the intersection $F_{n-1} \cap \mathcal{A}(PB_{n})$ is amenable and normal in $F_{n-1}$. It means $F_{n-1} \cap \mathcal{A}(PB_{n})$ is trivial.
Now we have $$ 1 \to \mathcal{A}(PB_{n}) \overset{p|_{\mathcal{A}(PB_{n})}}{\to} p(\mathcal{A}(PB_{n})) \to 1, $$
and equivalently, $\mathcal{A}(PB_{n})$ is isomorphic to $p(\mathcal{A}(PB_{n}))$ but $p(\mathcal{A}(PB_{n})) \subset \mathcal{A} (PB_{n-1}) = \ZZ$ so $p(\mathcal{A}(PB_{n}))$ is isomorphic to either $1$ or $\ZZ$.
 But the center of $PB_{n}$ is $\ZZ$ and the amenable radical should contain the center. Therefore, $\mathcal{A}(PB_{n})$ is isomorphic to $\ZZ$.

 Now we will prove that $\mathcal{A}(PB_{\infty})$ is trivial. Suppose not, and choose $1 \neq x \in \mathcal{A}(PB_{\infty})$.
 By the definition, $x \in PB_k$ for some $k$. We choose $k$ as minimal so we may assume $x \not \in PB_l$ for all $l<k$ 
 but $x \in PB_l$ for all $l \geq k$. Note that for any $n \in \NN$, the intersection $ \mathcal{A}(PB_{\infty}) \cap PB_n $ is an amenable normal subgroup of $PB_n$. 
 Then, for any $m \geq k$, the intersection $$ \mathcal{A}(PB_{\infty}) \cap PB_m $$ is an amenable normal subgroup of $PB_m$ containing $x$.
 By the claim above, the intersection is isomorphic to $\ZZ$ since the intersection contains $x \neq 1$.

 Consider the intersection $$ \mathcal{A}(PB_{\infty}) \cap PB_{k+2}.$$
 Then it contains $x$ and $x \in PB_k$ so in braid representation, $x$'s $k+1$ and $k+2$ strings are just vertical lines.
 Since the intersection is isomorphic to $\ZZ$ and it contains the center $Z(PB_{k+2})$, $x^N \in Z(PB_{k+2})$ for some $N$. 
 It is a contradiction. Therefore, $\mathcal{A}(PB_{\infty})$ is trivial and $\mathcal{A}(G)$ is also trivial.
\end{proof}

 These lemmas can be summarized as follows.

\begin{prop}
 Let $G:=BV$ be the braided Thompson group and $R$ be the ray graph. Then $K(G)$ exists and $BV$ satisfies the equalities 
 $$ \ker \left ( G \curvearrowright \Cone_{\omega}(G,d_n) \right ) = \ker \left( G \curvearrowright \partial R \right ) = K(G) = \FC(G) = \mathcal{A}(G) = 1 $$ 
 hold for all ultrafilter $\omega$ and sequence $d_n$.
 However, $BV$ is not acylindrically hyperbolic.
\end{prop}

It is still unknown whether the group action of $BV$ on the ray graph $R$ contains two independent loxodromic. Motivated by the above result and question, we were wondering if the following is true.

\begin{ques}
 Let $G$ be a finitely generated group and suppose that $G$ acts isometrically on some $\delta$-hyperbolic space $X$ and $G$ contains two independent loxodromic elements. If $K(G)$ exists and the equalities 
 $$ \ker \left ( G \curvearrowright \Cone_{\omega}(G,d_n) \right ) = \ker \left( G \curvearrowright \partial X \right ) = K(G) = \FC(G) =  \mathcal{A}(G) $$ 
 hold for all ultrafilter $\omega$ and sequence $d_n$, then $G$ is acylindrically hyperbolic?
\end{ques}

 If the action of $BV$ on the ray graph admits two independent loxodromic elements, then the answer is automatically false. If the question is answered affirmatively, then this gives another characterization for being acylindrically hyperbolic. 
 
%%%%%%%%%%%%%%%%%%%%%%%%%%%%%%%%%%%%%%%%%%%%%%%%%%%%%%%%%%%%%%%%%%%%%%

\bibliography{preprint}
\bibliographystyle{abbrv}

\end{document}